\documentclass[12pt]{article}

\usepackage{amsfonts}
\usepackage{amssymb}
\usepackage{amsmath}
\usepackage{amsthm}
\usepackage{graphicx}
\usepackage{float}
\usepackage{citesort}

\newcommand\RR {\mathbb{R}}
\newcommand\NN {\mathbb{N}}
\newcommand\ZZ {\mathbb{Z}}
\newcommand\CC {\mathbb{C}}

\newcommand{\abs}[1]{\left |#1\right|}
\newcommand{\norm}[1]{\left \lVert#1 \right \rVert}

\newcommand{\partialDerivative}[2]{\frac{\partial{#1}} {\partial{#2}} }

\theoremstyle{plain}
\newtheorem{theorem}{Theorem}
\newtheorem{lemma}{Lemma}

\theoremstyle{remark}
\newtheorem{remark}{Remark}

7.0in \title{Rapidly convergent quasi-periodic Green functions for
  scattering by arrays of cylinders---including Wood anomalies}
\author{Oscar P. Bruno\footnote{Computing and Mathematical Sciences,
    Caltech, Pasadena, CA 91125, USA} \and Agustin G. Fernandez-Lado$^*$}
\date{}

\begin{document}
\maketitle
\begin{abstract}
  This paper presents a full-spectrum Green function methodology
  (which is valid, in particular, at and around Wood-anomaly
  frequencies) for evaluation of scattering by periodic arrays of
  cylinders of arbitrary cross section---with application to wire
  gratings, particle arrays and reflectarrays and, indeed, general
  arrays of conducting or dielectric bounded obstacles under both TE
  and TM polarized illumination. The proposed method, which, for
  definiteness is demonstrated here for arrays of perfectly conducting
  particles under TE polarization, is based on use of the shifted
  Green-function method introduced in the recent contribution (Bruno
  and Delourme, Jour. Computat. Phys. pp. 262--290 (2014)). A certain
  infinite term arises at Wood anomalies for the cylinder-array
  problems considered here that is not present in the previous
  rough-surface case. As shown in this paper, these infinite terms can
  be treated via an application of ideas related to the
  Woodbury-Sherman-Morrison formulae. The resulting approach, which is
  applicable to general arrays of obstacles even at and around
  Wood-anomaly frequencies, exhibits fast convergence and high
  accuracies. For example, a few hundreds of milliseconds suffice for
  the proposed approach to evaluate solutions throughout the
  resonance region (wavelengths comparable to the period and cylinder
  sizes) with full single-precision accuracy.
\end{abstract}

\section{Introduction}
\label{sec:Introduction}
We consider the problem of scattering of a monochromatic plane wave by
a periodic array of cylinders of general cross section. We approach
this problem by means of the methodology introduced
in~\cite{BrunoDelourme} which, based on use of a certain shifted Green
function, provides a solver for problems of scattering by periodic
surfaces which is valid and accurate throughout the
spectrum---including Wood frequencies~\cite{Rayleigh,Wood}, at which
the classical quasi-periodic Green function ceases to exist

A variety of approaches have been used to tackle this important
problem including, notably, methods based on use of integral
equations;
cf.~\cite{BarnettGreengard1,BarnettGreengard2,BrunoDelourme,BrunoHaslam1,BrunoHaslam2}
and references therein. The success of the integral-equation approach
results from its inherent dimensionality reduction (only the boundary
of the domain needs to be discretized) and associated automatic
enforcement of radiation conditions.

For the sake of simplicity the methodology presented in this article
assumes perfectly conducting obstacles under TE polarization. The
method can be easily extended to TM polarization and dielectric
cylinders: the dielectric case does not give rise to additional
difficulties, and the hyper-singular operators that arise in the TM case
of polarization can be handled by means of existing regularization
techniques (see e.g.~\cite{BrunoEllingTurc} and references therein). The main
strategy developed in the present paper can thus be applied in those
contexts without significant modifications.

As is well known, classical expansions for quasi-periodic Green
functions converge extremely slowly, and they of course completely
fail to converge at Wood anomalies. A number of methods have been
introduced to tackle the slow-convergence difficulty, including the
well known Ewald summation method for two and three dimensional
problems~\cite{Arens,Capolino,Linton1,Linton2} and many other
contributions~\cite{ChandlerWilde,Dienstfrey,KurkcuReitich,MathisPeterson,Moroz,NicoroviciMcPhedran1,NicoroviciMcPhedran2}. Unfortunately,
however, none of these methods resolve the difficulties posed by Wood
anomalies. Recently, a new quasi-periodic Green function was
introduced~\cite{BrunoDelourme} for the problem of scattering by
periodic surfaces which, relying on use of certain linear combinations
of shifted free-space Green functions (which amount to discrete
finite-differencing of the Green functions) can be used to produce
arbitrary (user-prescribed) algebraic convergence order for
frequencies throughout the spectrum, including Wood
frequencies~\cite{Rayleigh,Wood}.

A straightforward application of this procedure leads to an operator
equation that contains denominators which tend to zero as a Wood
anomaly is approached. To remedy this situation a strategy based on
use of the Woodbury-Sherman-Morrison formulae is introduced which
completely regularizes the problem and provides a limiting solution as
Wood frequencies are approached. To our knowledge, this is the first
approach ever presented that is applicable to problems of scattering
by periodic arrays of bounded obstacles at Wood anomalies on the basis
of quasi-periodic Green functions. It is worth mentioning that an
alternative method, not based on the use of quasi-periodic Green
functions and which is also applicable at Wood anomalies, was proposed
in~\cite{BarnettGreengard1,BarnettGreengard2}. In that approach, whose
generalization to corresponding three-dimensional problems at Wood
anomalies has not been provided, the quasi-periodicity is enforced
through use of auxiliary layer potentials on the boundaries of the
periodic cell. As suggested by the treatment~\cite{BrunoShipman} for
the problem of scattering by bi-periodic surfaces in three-dimensions,
the shifted Green function approach can be extended to
three-dimensional problems without difficulty.

In order to demonstrate the character of the new approach we present
numerical methods based on use of a combination of three main
elements: the half-space quasi-periodic Green function, the smooth
windowing methodology~\cite{BrunoShipman,BrunoDelourme,Monro} (which
gives rise to super-algebraically fast convergence away from Wood
anomalies) and high-order quadratures for singular
integrals~\cite{kressLIE,Kussmaul,Martensen,kressInverse}.  As shown
by means of a variety of numerical results, highly accurate solvers
result from this strategy---even at and around Wood anomalies.  As can
be seen in Section~\ref{sec:numericalResults} the present approach can
solve the complete scattering problem for scatterers of arbitrary
shape even at anomalous configurations, in fast computing
times.

The remainder of this article is organized as
follows. Section~\ref{sec:Preliminaries} presents necessary background
on the problem of scattering by periodic
media. Section~\ref{sec:ShiftedGreenFunction} summarizes the
convergence properties of the shifted quasi-periodic Green function
approximation introduced in~\cite{BrunoDelourme}, and
Section~\ref{sec:formul} then presents an associated integral-equation
formulation for the solution of the problem of wave scattering by a
periodic array of bounded obstacles at all frequencies.  The actual
numerical implementation we propose, which relies, in addition, on the
aforementioned smooth windowing methodology, is presented in
Section~\ref{sec:numericalAlgorithm} (the super-algebraic convergence
of the smooth windowing methodology at non-anomalous configurations
has been established in~\cite{BrunoShipman} for bi-periodic structures
in three dimensional space; the same ideas can be applied to establish
the result for a linear periodic array in two-dimensional case).
Section~\ref{sec:numericalResults}, finally, presents a variety of
numerical results demonstrating the properties of the overall
approach.
\section{Preliminaries}
\label{sec:Preliminaries}
\subsection{Scattering by a periodic array of bounded obstacles}\label{sec:array_bdd}
We consider problems of scattering by a one-dimensional
perfectly-conducting diffraction grating of period $L$ ($L>0$) in
two-dimensional space under TE polarization. The scatterer is assumed
to equal a union of the form
\begin{equation}\label{array}
  D=\bigcup_{n = -\infty}^\infty D_n
\end{equation}
of mutually disjoint closed sets $D_n$, where $D_0$ equals a union of
a finite number of non-overlapping connected bounded components, and
where $D_n=D_0+n L \hat{x}$, ($\hat{x} = (1,0)$, $n\in \ZZ$); see
Figure~\ref{problemGeometry}.  We assume the grating is illuminated by
a plane-wave of spatial frequency $k>0$ and incidence angle $\theta$
(measured from the $y$-axis). The scattered wave $u^{\mathrm{s}}$
satisfies the PDE
\begin{equation}\label{PDE}
\Delta u^\mathrm{s}+k^2u^\mathrm{s} = 0 \qquad \mathrm{in}\quad \Omega =\RR^2 \setminus D
\end{equation}
along with the condition of radiation at infinity (see
Remark~\ref{rem:Radiating}) and a Dirichlet-type boundary condition:
\begin{equation}\label{TE}
  u^{\mathrm{s}}=-u^{\mathrm{inc}}_k\quad\mbox{on}\quad \partial  D
\end{equation}
where, letting $\alpha=k\sin(\theta)$ and
$\beta=k\cos(\theta)$ we have set
$$u^{\mathrm{inc}}_k(x,y)=e^{i\alpha x-i\beta y}.$$ 
\begin{figure}[h!]
 \centering
    \includegraphics[scale=0.6]{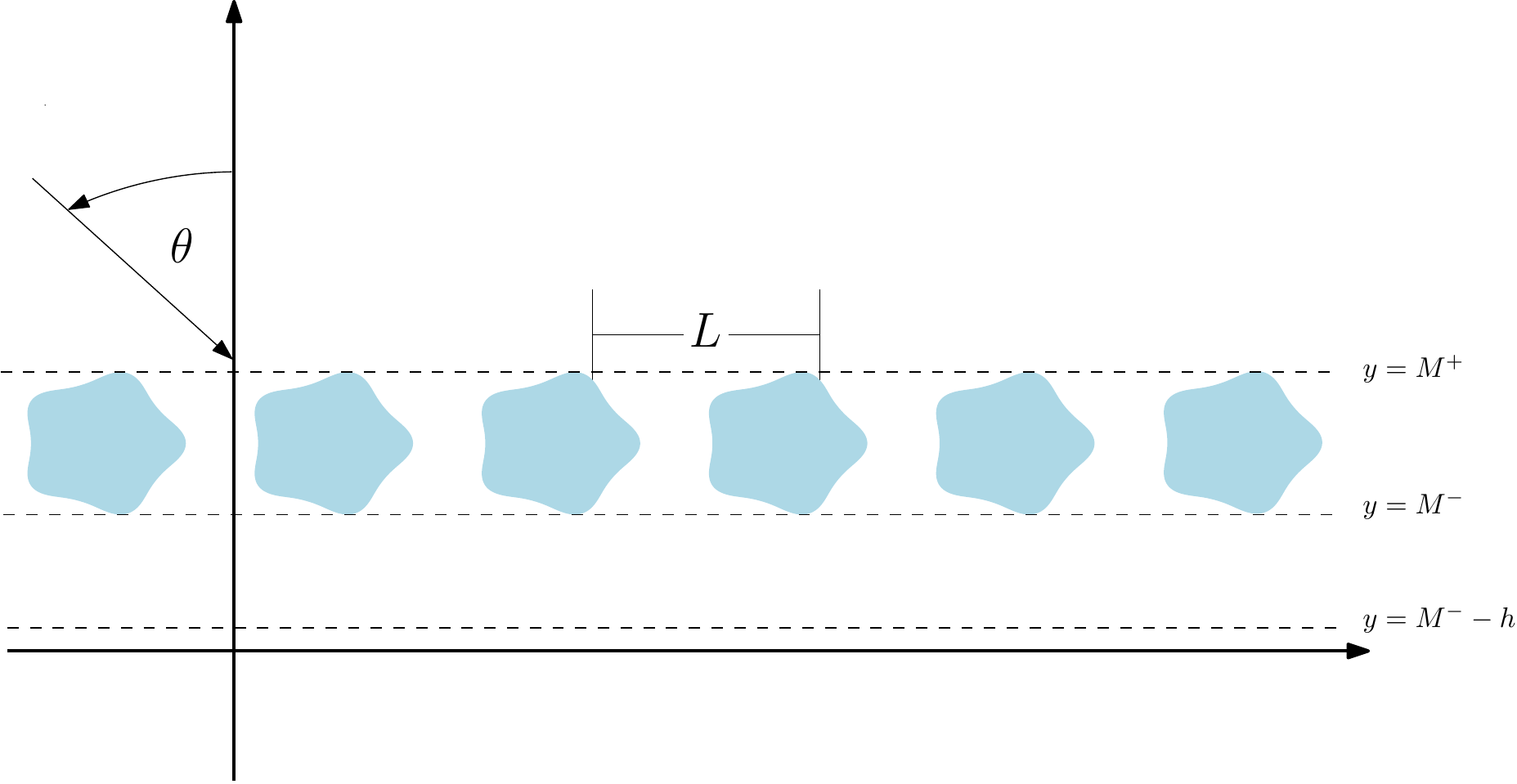}
    \caption{The physical problem: a plane wave with angle of
      incidence $\theta$ is scattered by an array of cylindrical
      obstacles distributed periodically in the $x$-direction. The
      particular star-shaped cylinder cross-section depicted here is
      given by the parametrization $C(t)= r(t) (\cos(t),\sin(t)) \quad
      \mbox{where} \quad
      r(t)=1+0.1\cos(5t)+0.01\cos(10t)$. \label{problemGeometry}}
\end{figure}

As is well known, the incoming and scattered waves $u^{\mathrm{inc}}_k$ and
$u^{\mathrm{s}}$ are $\alpha$ quasi-periodic functions, that is, they satisfy
the identities
$$u^{\mathrm{inc}}_k(x+nL,y)=e^{i\alpha nL }u_{\mathrm{inc},}(x,y)\quad\mbox{and}\quad u^{\mathrm{s}}(x+nL,y)=e^{i\alpha nL}u^{\mathrm{s}}(x,y).$$
It follows that, assuming
\begin{equation}\label{bounded}
  D \subset  \RR \times (M^-,M^+)
\end{equation} 
(see Figure~\ref{problemGeometry}) and letting $\alpha_n=\alpha+\frac{2\pi}{L}n$,
$\beta_n^2=k^2-\alpha_n^2$, ($\mathrm{Im}(\beta_n)\geq0$), the
solution $u^{\mathrm{s}}$ admits \textit{Rayleigh expansions}~\cite{petit} of the
form
\begin{equation}
  u^{\mathrm{s}}(x,y)=\sum \limits_{n\in \ZZ} A_n^\pm e^{i\alpha_n x +i\beta_n y}+B_n^\pm e^{i\alpha_n x-i\beta_n y}.
\end{equation}
\begin{remark}
\label{rem:Radiating}
The scattered field consists of a superposition of plane waves that
propagate away from $D$ and which remain bounded in the far field; a
scattered solution having this property is called
\textit{radiating}. In the present context we impose such a radiation
condition by requiring that~\cite{petit}
\begin{equation}\label{rayleigh_plus}
  u^{\mathrm{s}}(x,y)=\sum \limits_{n\in \ZZ} C_n^+ e^{i\alpha_n x +i\beta_n y}, \quad y>M^+,\quad\mbox{and}
\end{equation}

\begin{equation}\label{rayleigh_minus}
  u^{\mathrm{s}}(x,y)=\sum \limits_{n\in \ZZ} C_n^- e^{i\alpha_n x-i\beta_n y}, \quad y<M^-.
\end{equation}

\end{remark}

\begin{remark}\label{rem:WA_def}
  Throughout this paper $\mathcal{N}$ denotes the finite subset of
  integers for which $k^2-\alpha_n^2>0$. For $n \in \mathcal{N}$ the
  functions $e^{i\alpha_n x+i\beta_n y}$ and $e^{i\alpha_nx-i\beta_n
    y}$ are outwardly propagating waves (above and below $D$,
  respectively). If $k^2-\alpha_n^2<0$, the modes $e^{i\alpha_n x +
    i\beta_n y}$ (resp. $e^{i\alpha_n x - i\beta_n y}$) are evanescent
  waves, i.e., they decrease exponentially as $y\to +\infty$ (resp. $y
  \to -\infty$). A scattering setup for which $k^2=\alpha_{n_0}^2$ for
  some $n_0\in \ZZ$ is called a \textit{Wood anomaly configuration} and the wavenumber $k$ will be referred as a \textit{Wood frequency} or \textit{Wood anomaly frequency};
  note that, in such cases the plane wave $e^{i\alpha_{n_0} x \pm
    i\beta_{n_0} y}$ travels in directions parallel to the obstacle
  array. For fixed $L$ and $\theta$ we will consider the set
\begin{equation}\label{KWA}
\mathcal{K}_{\mathrm{wa}} = \mathcal{K}_{\mathrm{wa}}(L,\theta) =  \{ k>0 :\, \exists\, n\in \ZZ \text{ such that }  k^2=\alpha_n^2\}
\end{equation} 
of all Wood frequencies for a given period $L$ and incidence angle
$\theta$.  Clearly, for each
$k_0 \in \mathcal{K}_{\mathrm{wa}}$ the set
\begin{equation}
\mathcal{N}_{\mathrm{wa}}(k_0)=\left\{n\in \ZZ: \beta_n(k_0)=0 \right\}
\end{equation}
has at least one and at most two elements: either
$\mathcal{N}_{\mathrm{wa}} = \{ n_1\}$ or $\mathcal{N}_{\mathrm{wa}} =
\{ n_1,n_2\}$. For conciseness, throughout this paper we will write
$\mathcal{N}_{\mathrm{wa}}= \{ n_1,\dots,n_r\}$ with $r=1$ or $r=2$,
as appropriate, with corresponding expressions such as $\sum_{j=1}^r$,
etc.---in spite of the fact that the set and the sum, etc., can only
have one or two elements.
\end{remark}

The \textit{energy balance} relation for the type of geometrical
arrangement under consideration, which follows easily from
consideration of Green's identities (cf.~\cite{petit}), is given by
\begin{equation}\label{engy_bal_eqn}
\mathrm{Re}(C_0^-)+\sum_{n\in \mathcal{N}} e_n^+ + e_n^-=0.
\end{equation}	
where
\begin{equation}
  e_n^\pm=\abs{C_n^\pm}^2\frac{\beta_n}{\beta}\quad n \in \mathcal{N}.
\end{equation}

\subsection{Classical quasi-periodic Green function}

Given $k\not \in \mathcal{K}_{\mathrm{wa}}$, the classical quasi-periodic Green function is defined by 
\begin{equation}
\label{classicalQuasiGreen}
  G^q_{0,k}(X,Y)=\sum \limits_{n\in \ZZ}e^{-i\alpha n L} G_{0,k}(X+nL,Y)
\end{equation}
where $G_{0,k}$ denotes the free space Green function for the Helmholtz equation 
\begin{equation}
  G_{0,k}(X,Y)=\frac{i}{4}H_0^{(1)}(k\sqrt{X^2+Y^2})
\end{equation}
and $H_0^{(1)}$ is the zero-order Hankel function of the first
kind. (The subindex ``0'' in $G_{0,k}$ and $G^q_{0,k}$ indicates that
the number of ``shifts'' used is zero;
cf. Section~\ref{sec:ShiftedGreenFunction} below.) This infinite sum
converges at all points $(X,Y)\neq(md,0)$, $m\in \ZZ$ and, moreover,
the truncated series
\begin{equation*}
  \sum \limits_{n\in \ZZ, \abs{n}\geq 2} e^{-i\alpha n L} G_{0,k}(X+nL,Y)
\end{equation*}
converges uniformly over compact subsets of the plane not containing singularities of $G^q_{0,k}$ \cite{BrunoReitich}. 

The lack of convergence of the classical quasi-periodic Green
function~\eqref{classicalQuasiGreen} at Wood anomalies has historically presented some
of the most significant challenges in the solution of periodic
scattering problems. The following section introduces a modified
quasi-periodic Green function, namely the shifted quasi-periodic Green
function, which does not suffer from this difficulty.

\section{Shifted quasi-periodic Green function}
\label{sec:ShiftedGreenFunction}
In order to solve scattering problems at all frequencies we make use
of the shifted quasi-periodic Green function introduced
in~\cite{BrunoDelourme} which converges even at and around
Wood-anomaly frequencies. In this method, convergence at Wood
anomalies (and, indeed, fast convergence even away from Wood
anomalies) is achieved via addition of a number $j$ of new Green
function poles positioned outside the physical propagation domain,
along a certain given direction (e.g. vertically underneath the true
pole) and at distances $h$, $2h$, $\dots jh$ from the true Green
function pole. For example, taking $j=1$ the respective shifted Green function is
given by
\begin{equation}
G_{1,k}(X,Y)=G_{0,k}(X,Y)-G_{0,k}(X,Y+h).
\end{equation}
In order to appreciate the advantages that result from use of this
Green function we use the mean value theorem together with the
relation $(H_0^{(1)})^\prime=-H_1^{(1)}$ and obtain
\begin{equation}
  G_{1,k}(X,Y)=\frac{i}{4}\frac{hk(Y+\xi)}{\sqrt{X^2+(Y+\xi)^2}}H_1^{(1)}\left(k\sqrt{X^2+(Y+\xi)^2}\right)
\end{equation}
for some $\xi\in (0, h)$.  In view of the asymptotic expression
\begin{equation}
H_n^{(1)}(t)=\sqrt{\frac{2}{\pi t}} e^{i\left(t-n\frac{\pi}{2}-\frac{\pi}{4}\right)}\left\{ 1+O \left(\frac{1}{t}\right)\right\}, \quad  t \in \RR, \quad n \in \NN_0,
\end{equation}
it follows that there exists $C>0$ such that for large values of
$\abs{X}$ and bounded $\abs{Y}$ we obtain the enhanced decay
\begin{equation}
\abs{G_{1,k}(X,Y)} \leq \frac{C}{\abs{X}^{3/2}}.
\end{equation}

As shown in~\cite{BrunoDelourme}, Green functions with arbitrary
algebraic decay can be obtained by generalizing this idea: given $h>0$
and $j\in \NN$ the half-space shifted Green function $G_{j,k}$ containing
$j$ shifts is defined by ``finite-differencing'' Green functions
associated with various poles:
\begin{equation}\label{Gj_def}
  G_{j,k}(X,Y)=\frac{i}{4} \sum \limits_{\ell=0}^j (-1)^\ell {j \choose \ell} H_0^{(1)}\left(k\sqrt{X^2+(Y+\ell h)^2}\right)
\end{equation}
for $(X,Y) \in \RR^2$, $(X,Y) \neq (0,\ell h), \, \ell=0 \dots
j$. (Note the $j$-th order finite-difference coefficients $(-1)^l {j
  \choose \ell}$ in equation~\eqref{Gj_def}). The fast decay of this
function is established by the following lemma whose proof can be
found in~\cite{BrunoDelourme}.
\begin{lemma}
  Let $j\in \NN$, $h>0$, $\ell=0$ or $\ell=1$, $m=1$ or $m=2$ and $k>0$ be
  given. Then, for each $M>0$ there exists a positive constant $C_M$
  (that also depends on $j,k$ and $h$) such that for all $y \in
  (-M,M)$ and for all real numbers $X$ with $\abs{X}>1$ we have
\begin{equation}
\abs{\partial_m^\ell G_{j,k}(X,Y)} \leq 
	\begin{cases}
		{C_M\abs{X}^{-\frac{j+1}{2}}}  &  j \text{ is even}\\
		{C_M \abs{X}^{-\frac{j}{2}-1}} & j \text{ is odd}
	\end{cases}
\end{equation}
where $\partial_m^\ell$ denotes differentiation of order $\ell$ in the $m$-th coordinate direction.
\end{lemma}

The corresponding shifted quasi-periodic Green function is thus defined by
\begin{equation}
\label{quasiPerGj}
  G_{j,k}^q(X,Y) = \sum \limits_{n \in \ZZ} e^{-i\alpha n L} G_{j,k}\left(X+nL,Y\right).
\end{equation}
The fast convergence of such series, which follows from the previous
lemma, is laid out in the following theorem.

\begin{theorem}
\label{thm:shiftedGreenFunctionDecay}
  Let $j\in \NN$, $h>0$, $\ell=0$ or $\ell=1$, $m=1$ or $m=2$ and $k_0>0$ be given. Then, for each $M>0$ and $\delta>0$, there exists a constant $D_{M,\delta}>0$ (that also depends on $j,k_0$ and $h$) such that, for all $X,Y$ satisfying $-L \leq X \leq L$ and $-M <Y <M$, for all frequencies $k>0$ such that $\abs{k-k_0}<\delta$, and for all integers $N>1$, we have
\begin{equation}
  \abs{\sum \limits_{n\in\ZZ: \abs{n}\geq N} e^{-i\alpha n L}\partial_m^\ell G_{j,k}(X+nL,Y)} \leq 
	\begin{cases}
		{D_{M,\delta}\abs{N}^{-\frac{j-1}{2}}}  &  j \text{ is even}\\
		{D_{M,\delta} \abs{N}^{-\frac{j}{2}}} & j \text{ is odd}
	\end{cases}.
\end{equation}
It follows that for $\ell=0$ and $\ell=1$ with $m=1$ or $m=2$:
\begin{enumerate}
\item The finite sums $\sum \limits_{\abs{n}\leq N} e^{-i\alpha n L}\partial_m^\ell G_{j,k}(X+nL,Y)$ converge as $N \to \infty$ to the corresponding quantities $\partial_m^\ell G^q_{j,k}$.
\item The corresponding approximation errors decrease at least as fast as $\abs{N}^{-\frac{j-1}{2}}$ if $j$ is even and as fast as $\abs{N}^{-\frac{j}{2}}$ if $j$ is odd.
\end{enumerate}
\end{theorem}
\begin{proof}
  This result, which is equivalent to Theorem 4.4
  in~\cite{BrunoDelourme}, emphasizes the fact that the constant
  $D_{M,\delta}$ can be taken to be independent of the wavenumber $k$
  for all $k$ in a neighborhood of a given wavenumber $k_0$. The proof
  of the present version of the result can be obtained by inspection
  of the corresponding proof in~\cite{BrunoDelourme}.
\end{proof}
\section{Periodic array scattering at and around Wood anomalies}
\label{sec:formul}
\subsection{Strategy}
As is well known~\cite{BrunoReitich,Linton1,Linton2}, the classical
quasi-periodic Green function displayed in
equation~\eqref{classicalQuasiGreen} can alternatively be expressed in
the spectral form
\begin{equation}
  \label{classicalQuasiGreenSpectral}
  G^q_{0,k}(X,Y)=\frac{i}{2L}\sum \limits_{n\in \ZZ} \frac{e^{i\alpha_nX+i\beta_n \abs{Y}}}{\beta_n}
\end{equation}
which, of course, is only valid provided $\beta_n \neq 0$ for all
$n\in \ZZ$. In other words,
equation~\eqref{classicalQuasiGreenSpectral} is only meaningful away
from Wood anomalies (see Remark~\ref{rem:WA_def}).  Accordingly, the
``spatial'' series~\eqref{classicalQuasiGreen} fails to converge at
Wood-anomaly frequencies, and its convergence slows down significantly
as a Wood anomaly is approached. In the context of
equation~\eqref{classicalQuasiGreenSpectral} this difficulty clearly
arises from those terms in
equation~\eqref{classicalQuasiGreenSpectral} whose denominator tends
to zero as a Wood anomaly is approached: if such ``Wood'' terms are
excluded then the infinite sum~\eqref{classicalQuasiGreenSpectral}
converges and, in fact, it yields an analytic function with respect to
$k$ for each $(X,Y)\ne (m L,0)$, $m\in\ZZ$. The solution strategy
described in the present Section~\ref{sec:formul}, which is based on
detailed analysis around the individual Wood terms, can be briefly
summarized as follows:

\begin{enumerate}
\item\label{one} Integral equations for array-scattering based on the classical
  Green function~\eqref{classicalQuasiGreen} are obtained;
\item\label{two} The resulting integral operator is
  re-expressed as an integral operator defined in terms of the shifted
  Green function (which is defined for every frequency,
  including Wood anomalies), plus a certain ``Rayleigh series
  operator'';
\item For a given Wood-anomaly frequency $k_0\in
  \mathcal{K}_\mathrm{wa}$, the Rayleigh series operator mentioned in
  point~\ref{two} includes a finite-rank operator which encapsulates
  the singular character of the Wood-anomaly problem. Solutions for
  the resulting linear operator equation can then be obtained at and
  around $k_0$ by resorting to ideas closely related to the
  Woodbury-Sherman-Morrison formulae.
\end{enumerate}
Points 1. and 2. are considered in Section~\ref{sec:Pot+intEq} while
point 3. is addressed in Section~\ref{sec:Woodbury}.
\subsection{Hybrid ``Rayleigh-Expansion/Integral-Equation''
  formulation}
\label{sec:Pot+intEq}
Given $k\not \in \mathcal{K}_{\mathrm{wa}}$ the potential 
\begin{equation}\label{solution}
  u^{\mathrm{s}}_k(r)=\int\limits_{\partial D_0}\left(  \partialDerivative{G_{0,k}^q}{\nu(r^\prime)}-i\gamma G_{0,k}^q\right) (r-r^\prime)\psi(r^\prime) dS(r^\prime)
\end{equation}
is a quasi-periodic, radiating solution of the Helmholtz equation in
$\Omega=\RR^2\setminus D$. Note that, in view of the quasi-periodicity
of $u^{\mathrm{s}}$ and $u^\mathrm{inc}$, $u^{\mathrm{s}}$ satisfies
the boundary conditions~\eqref{TE} if and only if it satisfies the
boundary conditions
\begin{equation}\label{TE0}
  u^{\mathrm{s}}=-u^{\mathrm{inc}}_k\quad\mbox{on}\quad \partial  D_0
\end{equation}
where $\partial D_0$ is the portion of the scattering boundary
contained the unit-cell.  Using the well known jump relations of the
single and double layer
potentials~\cite{kressLIEScattering,kressInverse} it follows that
$u^{\mathrm{s}}_k$ is a solution to equations~\eqref{PDE} and
\eqref{TE} if and only if $\psi \in C(\partial D_0)$ satisfies the
integral equation
\begin{equation}\label{intEqG0}
  \frac{1}{2}\psi(r)+\int\limits_{\partial D_0}\left(  \partialDerivative{G_{0,k}^q}{\nu(r^\prime)}-i\gamma G_{0,k}^q\right)(r-r^\prime) \psi(r^\prime) dS(r^\prime)=-u^{\mathrm{inc}}_k(r), \quad r\in \partial D_0.
\end{equation}
Here $\gamma$ is a non-negative real number and the linear combination of
single and double layer potentials in equation~\eqref{solution} is
used to ensure invertibility of the integral-equations formulation at
wavenumbers that equal eigenvalues of the Laplace operator within
$D_0$; see e.g.~\cite{kressLIEScattering} as well as the related
literature~\cite{BrakhageWerner,Leis,Panich,kressInverse}.

 An alternative operator equation can be obtained by noting that
\begin{equation}
G^q_{0,k}(X,Y)=G^q_{j,k}(X,Y) - \sum \limits_{\ell=1}^j (-1)^\ell {j \choose \ell} G_{0,k}^q(X,Y+\ell h)
\end{equation}
and that, using~\eqref{classicalQuasiGreenSpectral}, for $Y>-h$ the
rightmost sum in the above equation can be expressed in the form
\begin{equation}
\frac{i}{2L}\sum \limits_{\ell=1}^j (-1)^\ell {j \choose \ell} G_{0,k}^q(X,Y+\ell h)=\sum \limits_{n\in \ZZ} \frac{\sigma_n(k)}{\beta_n(k)}  e^{i\alpha_n(k) X+ i\beta_n(k) Y}
\end{equation}
where $\sigma_n(k)=(1-e^{i\beta_n(k) h})^j -1$. Thus, selecting a
sufficiently large shift spacing,
\begin{equation}
\label{eq:Assumption}
h>\max\{\abs{y-y^\prime}:(x,y),(x^\prime,y^\prime) \in \partial D_0 \mbox{ for some } x,x^\prime\},
\end{equation}
the integral operator in equation~\eqref{intEqG0} equals
\begin{equation}
\label{operatorGj}
  \int \limits_{\partial D_0}
  \left(\partialDerivative{G_{j,k}^q}{\nu(r^\prime)}-i \gamma G_{j,k}^q\right)
  (r-r^\prime) \psi(r^\prime) dS(r^\prime)-\frac{i}{2L}\sum \limits_{n \in \ZZ}
    \sigma_n(k) \frac{I_{n,k}^+\left[\psi\right]}{\beta_n(k)} e^{i\alpha_n(k) x +i\beta_n(k) y}
\end{equation}
where
\begin{equation}
\label{InPlus}
I_{n,k}^+\left[\psi\right]=\int \limits_{\partial D_0} \left( \partialDerivative{}{\nu(r^\prime)} \left( e^{-i\alpha_n(k) x^\prime -i\beta_n(k) y^\prime}\right) -i\gamma e^{-i\alpha_n(k) x^\prime -i\beta_n(k) y^\prime}\right) \psi(r^\prime)dS(r^\prime).
\end{equation}
It follows that, letting $(x,y) = r \in \partial D_0$ denote the
Cartesian coordinates of $r$, equation~\eqref{intEqG0} is equivalent
to
\begin{equation}
\label{intEqGj}
  \frac{1}{2} \psi(r) + \int \limits_{\partial D_0}
  \left(\partialDerivative{G_{j,k}^q}{\nu(r^\prime)}-i \gamma G_{j,k}^q\right)
  (r-r^\prime) \psi(r^\prime) dS(r^\prime)-\frac{i}{2L}\sum \limits_{n \in \ZZ} \sigma_n(k) \frac{I_{n,k}^+\left[\psi\right]}{\beta_n(k)} e^{i\alpha_n(k) x +i\beta_n(k)
    y}=-u^{\mathrm{inc}}_k(r).
\end{equation}
The corresponding expression for the potential~\eqref{solution} in
terms of the shifted Green function $G_{j,k}^q$ is given by the expression
\begin{equation}
\label{eq:potRepGj+}
u^{\mathrm{s}}_k(r)=\int \limits_{\partial D_0}
  \left(\partialDerivative{G_{j,k}^q}{\nu(r^\prime)}-i \gamma G_{j,k}^q\right)
  (r-r^\prime) \psi(r^\prime) dS(r^\prime)-\frac{i}{2L}\sum \limits_{n \in \ZZ}  \sigma_n(k)\frac{I_{n,k}^+\left[\psi\right]}{\beta_n(k)} e^{i\alpha_n(k) x +i\beta_n(k)
    y},
\end{equation}
which is valid in the region $\Omega^+=\{y\geq M^--h\}$ (see
Figure~\ref{problemGeometry}). For points in $\Omega^-=\{y<M^--h\}$,
the potential is represented by the Rayleigh series
\begin{equation}
\label{eq:potRepGj-}
u^{\mathrm{s}}_k(r)=\frac{i}{2L} \sum \limits_{n\in \ZZ}  \frac{I_{n,k}^-\left[\psi\right]}{\beta_n(k)} e^{i\alpha_n(k) x -i\beta_n(k) y}
\end{equation}
where we have set
\begin{equation}
\label{InMinus}
I_{n,k}^-[\psi]=\int \limits_{\partial D_0} \left( \partialDerivative{}{\nu(r^\prime)} \left( e^{-i\alpha_n(k) x^\prime +i\beta_n(k) y^\prime}\right) -i\gamma e^{-i\alpha_n(k) x^\prime +i\beta_n(k) y^\prime}\right) \psi(r^\prime)dS(r^\prime).
\end{equation}

Away from Wood anomalies, the integral equation
formulation~\eqref{intEqG0} or, equivalently, the hybrid integral
equation/Rayleigh series operator formulation~\eqref{intEqGj} are well
posed: the same arguments used in the proof of the invertibility of
the combined field formulation for a single bounded obstacle can be
applied to the periodic problem provided that we assume uniqueness of
solutions for the periodic PDE problem. Once these operator equations
are solved, the representation formulas~\eqref{solution}
or~\eqref{eq:potRepGj+} and~\eqref{eq:potRepGj-} can be used to obtain
scattering solutions. 

However, at and around Wood anomalies further work is needed. As 
mentioned previously, the classical quasi-periodic Green function
ceases to exist at Wood frequencies and therefore the integral
equation~\eqref{intEqG0} cannot be used. Even though the shifted Green
function does exist at Wood anomalies (and, therefore, integral
operators that have it as kernel are well-defined) the hybrid
formulation~\eqref{intEqGj} cannot be used directly at this singular
case. Indeed, at Wood anomalies the infinite sum in
equation~\eqref{intEqGj} contains some vanishing denominators. The
merit of equation~\eqref{intEqGj}, however, is that it makes explicit
the singular behaviour at Wood anomalies in the form of a finite rank
operator (the finite sum of those terms whose denominator are close to
zero) while the remainder of the hybrid integral/Rayleigh series
operator in the left hand side of equation~\eqref{intEqGj} is well
defined. Thus, the strategy to solve the problem relies on being able
to invert equations of the general form
\begin{equation}
(A+R) \psi = f
\end{equation}
where $A$ is assumed to be an invertible operator and $R$ a
finite-rank operator containing vanishing denominators. The next
section introduces a general linear algebra result (which can be seen
as a slight re-interpretation of the Woodbury-Sherman-Morrison
formulae) and then applies it to equation~\eqref{intEqGj}.

\subsection{Solution at and around Wood-anomaly frequencies}
\label{sec:Woodbury}
Let $k_0 \in \mathcal{K}_{\mathrm{wa}}$ be a Wood frequency and
consider the set of integers $\mathcal{N}_{\mathrm{wa}}=\{n\in \ZZ:
\beta_n(k_0)=0\}$ which, in view of Remark~\ref{rem:WA_def}, has at
most two elements. For a frequency $k$ in a neighbourhood of $k_0$ we
define the finite-rank operators
\begin{align}\label{RDirichlet}
  R_{k}(\psi)(x,y)&= -\frac{i}{2L} \sum \limits_{n \in
    \mathcal{N}_{\mathrm{wa}}}  \sigma_n(k)
  \frac{I_{n,k}^+\left[\psi\right]}{\beta_n(k)} e^{i\alpha_n(k)
    x+i\beta_n(k) y}
  \\
\label{RDirichlet2}  \widetilde R_{k}(\psi)(x,y)&= -\frac{i}{2L} \sum \limits_{n \in \mathcal{N}_{\mathrm{wa}}}
  \sigma_n(k) I_{n,k}^+\left[\psi\right] e^{i\alpha_n(k)
    x+i\beta_n(k) y}
\end{align}
where $(x,y)$ are the Cartesian coordinates of a point in $\partial
D_0$. The operator $R_k$ isolates the finitely many terms in the
infinite sum in equation~\eqref{intEqGj} that contain denominators
that vanish as $k_0$ is approached. The operator $\tilde{R}_k$ is a
re-scaled version of $R_{k}$ that does not contain vanishing
denominators, and which will play a major role in the application of
the inversion formula~\eqref{eq:WoodburyFormula} to the scattering
problem under consideration. Equation~(\ref{intEqGj}) can then be
re-expressed in the form
\begin{equation}
\label{operatorEquation}
\left(A_k+R_k \right) \psi = -u^{\mathrm{inc}}_k
\end{equation}
where 
\begin{equation} 
\label{opA}
A_k\psi(r)=\frac{1}{2} \psi(r) + \int \limits_{\partial D_0} \left(\partialDerivative{G_{j,k}^q}{\nu(r^\prime)}-i \gamma G_{j,k}^q\right) (r-r^\prime) \psi(r^\prime) dS(r^\prime)-\frac{i}{2L} \sum \limits_{n \not \in \mathcal{N}_{\mathrm{wa}}}  \sigma_n(k) \frac{I_{n,k}^+\left[\psi\right]}{\beta_n(k)} e^{i\alpha_n(k) x+i\beta_n(k) y}.
\end{equation}
The proposed strategy to solve equation~\eqref{intEqGj} (or,
equivalently,~\eqref{operatorEquation}) at and around Wood anomalies
relies strongly on an operator formula, related to the Woodbury and
Sherman-Morrison relations~\cite[Sects. 2.7.1 and
2.7.3]{FortranRecipes}, which is presented in the following lemma.
\begin{remark}\label{rem_lemma}
  The notation and hypotheses used in the lemma are as follows: $X$
  denotes a general vector space over $\CC$, $\mathcal{L}(X)$ denotes
  the set of linear (not necessarily continuous) operators from $X$ to
  $X$, $A\in \mathcal{L}(X)$ is an invertible operator and
  $R_{\boldsymbol{b} }\in \mathcal{L}(X)$ denotes a finite rank
  operator which for $f\in X$ takes the value
\begin{equation}\label{f_rank}
  R_{\boldsymbol{b}}[f]= \sum \limits_{j=1}^r  \frac{\ell_j[f]}{b_j} w_j,
\end{equation}
where $b_j$ denote nonzero complex numbers, $\ell_j$ denote linear
functionals defined on $X$, and where $w_1, \dots , w_r$ form a
linearly independent set of elements of $X$.  Additionally, we
consider the re-scaled finite-rank operator
\begin{equation}
  R_{\boldsymbol{1}}[f] = \sum \limits_{j=1}^r \ell_j[f] w_j.
\end{equation}
Setting $\mathcal{W} =\{w_1, \dots , w_r \}$ and $S=\mathrm{span}
\left(\mathcal{W}\right )\subseteq X$ we see that both $R_{\boldsymbol{b}
}A^{-1}$ and $R_{\boldsymbol{1} }A^{-1}$ map all of $X$ into $S$ and,
in particular, they map $S$ into itself. We thus consider the
restrictions of $R_{\boldsymbol{b}}A^{-1}$ and
$R_{\boldsymbol{1}}A^{-1}$ to the subspace $S$ and denote them by
$T_{\boldsymbol{b}}$ and $T_{\boldsymbol{1}}$ respectively. Letting,
further, $D_{\boldsymbol{b}}:S\to S$ denote the diagonal operator
defined by
\begin{equation}\label{D_b}
  D_{\boldsymbol{b}} w_j =b_j w_j,
\end{equation}
we obtain $R_{\boldsymbol{b}}=D_{\boldsymbol{b}}^{-1}
R_{\boldsymbol{1}}$ throughout $X$ and thus, in particular,
$T_{\boldsymbol{b}}=D_{\boldsymbol{b}}^{-1}T_{\boldsymbol{1}}$.
\end{remark}
\begin{lemma}
\label{lem:Woodbury}
Using the notations and hypotheses as in Remark~\ref{rem_lemma},
assume that the operator $(D_{\boldsymbol{b}}+T_{\boldsymbol{1}}):S\to
S$ is invertible. Then the operator $(A+R_{\boldsymbol{b}})$ is also
invertible, and its inverse is given by
\begin{equation}
\label{eq:WoodburyFormula}
(A+R_{\boldsymbol{b}})^{-1} = A^{-1} \left( \mathbb{I} - (D_{\boldsymbol{b}}+T_{\boldsymbol{1}})^{-1} R_{\boldsymbol{1}}A^{-1} \right),
\end{equation}
where $\mathbb{I}$ denotes the identity operator.
\end{lemma}

\begin{proof}
  In view of the hypotheses of the lemma the operator on the
  right-hand side of equation~\eqref{eq:WoodburyFormula} is
  well-defined. The lemma follows by direct verification (via
  multiplication) that the right hand side operator
  in~\eqref{eq:WoodburyFormula} is both a left and right inverse of
  $A+R_{\boldsymbol{b}}$. (The expression~\eqref{eq:WoodburyFormula}
  for an operator $(A+R_{\boldsymbol{b}})$ in a normed space can be
  derived by writing $(A+z R_{\boldsymbol{b}})^{-1} =
  A^{-1}(\mathbb{I}+z R_{\boldsymbol{b}}A^{-1})^{-1}$, expressing the
  inverse operator $(\mathbb{I}+z R_{\boldsymbol{b}}A^{-1})^{-1}$ by
  means of a Neumann series which is convergent for $z$ sufficiently
  small, and performing some simple manipulations. The resulting
  formula holds for any value of $z$, and in particular for $z=1$.)
\end{proof}

The solutions of equation~\eqref{intEqGj}, for frequencies around Wood
anomalies, are obtained inverting the operator~\eqref{opA}; the
following Lemma~\ref{lem:contA_k} establishes a few regularity
properties of this operator which are necessary to establish
Theorem~\ref{thm:Woodbury}. In what follows $C(\partial D_0)$ denotes
the Banach space of complex-valued continuous functions along
$\partial D_0$, endowed with the maximum norm.
\begin{lemma}
\label{lem:contA_k}
Let $k_0 \in \mathcal{K}_{\mathrm{wa}}$. Then there exists $\delta>0$
such that, for $k \in (k_0-\delta,k_0+\delta)$, the operator $A_k$
maps $C(\partial D_0)$ into $C(\partial D_0)$ and the mapping $k \to
A_k$ is continuous.
\end{lemma}

\begin{proof}
  Let $\delta>0$, $\rho>0$ be such that $\abs{\beta_n(k)}>\rho$ for
  all $n\not \in \mathcal{N}_{\mathrm{wa}}(k_0)$ and all $k$
  satisfying $\abs{k-k_0}<\delta$. (The existence of such $\delta$ and
  $\rho$ values is easily checked from the definition of $\beta_n(k)$
  in Section~\ref{sec:array_bdd}.) Thus, for $\abs{k-k_0}<\delta$ all
  the denominators in the series in equation~\eqref{opA} are bounded
  away from zero. To study the mapping properties of $A_k$ we consider
  the integral operator and the ``Rayleigh-series'' operator in
  equation~\eqref{opA} separately.

  The kernel of the integral operator can be viewed as the sum of a
  weakly singular kernel and a continuous kernel. The weakly singular
  kernel is the portion of the term $n=0$ in the infinite
  sum~\eqref{quasiPerGj} which results by using $\ell=0$ in the
  corresponding finite sum $G_{j,k}$ given in equation~\eqref{Gj_def};
  the continuous kernel, in turn, equals the sum of all of the
  remaining terms in~\eqref{quasiPerGj}.  Clearly the weakly singular
  kernel equals a combination of the Hankel function $H_0^{(1)}\left(k
    | r|\right)$ ($|r| = \sqrt{X^2+Y^2}$) and its normal derivative.

  In view of~\cite[Thm 2.6]{kressLIEScattering} it follows that the
  integral operator in equation~\eqref{opA} maps continuous densities
  $\psi$ into continuous functions on $\partial D_0$.  Further, since
  the singular kernel can be expressed in the form $F_1\left(k |
    r|\right)\log \left(k | r|\right)+F_2\left(k | r|\right)$ with
  smooth functions $F_1$ and $F_2$, it is easy to check that the
  singular-term portion of the integral operator varies continuously
  with respect to $k$ (since $\log \left(k | r|\right)= \log
  \left(k\right) + \log \left(| r|\right)$ and since all other
  integrands in the singular-term operator vary smoothly with
  $k$). The continuity of the remaining terms of the integral operator
  in~\eqref{opA} as a function of $k$ follows easily from the uniform
  convergence, established in
  Theorem~\ref{thm:shiftedGreenFunctionDecay}, of the series that
  defines the shifted quasi-periodic Green function. Therefore, the
  whole integral operator is a continuous function of $k$ with values
  in the space $C(\partial D_0)$.

  The mapping properties of the ``Rayleigh-series'' operator, in turn,
  follow from the uniform convergence of the
  series~\eqref{operatorEquation}---whose terms, as shown in what
  follows, decay at a uniform exponential rate as soon as $n$ is
  sufficiently large that $\beta_n =i \abs{\beta_n}$. In detail, given
  $\psi \in C(\partial D_0)$ and $n$ sufficiently large, for an
  arbitrary point $(x,y) \in \partial D_0$ we obtain the following
  estimate
\begin{equation}\label{contA_k_eqn}
  \abs{\sigma_n(k) \frac{I_{n,k}^+[\psi]}{\beta_n(k)} e^{i\alpha_n(k) x+i\beta_n(k) y}}\leq \frac{(k+\gamma)}{\rho}\norm{\psi} \sum \limits_{\ell=1}^j  {j \choose \ell} \int \limits_{\partial D_0}  \abs{e^{-\abs{\beta_n(k)}(y- y^\prime+\ell h)}}dS(r^\prime)
\end{equation}
which, in view of the assumption~\eqref{eq:Assumption} on the shift
parameter $h$, clearly exhibits the claimed uniform exponential
decay. Since each one of the terms in the infinite series is a
continuous function of both the spatial variable along $\partial D_0$
as well as the frequency $k$, it follows that the infinite series
operator in~\eqref{opA} defines a continuous function along $\partial
D_0$ which varies continuously with $k$ in the space $C(\partial
D_0)$. This completes the proof of the Lemma.

\end{proof}

Theorem~\ref{thm:Woodbury} below, whose proof relies on
Lemmas~\ref{lem:Woodbury} and~\ref{lem:contA_k}, produces (unique)
solutions of equation~\eqref{intEqGj} for non-Wood frequencies
arbitrarily close to a Wood frequency, it shows that those solutions
admit a limit as the frequency tends to the Wood anomaly, and it
provides an explicit expression for the limit solution which does not
require use of a limiting process.

\begin{theorem}
\label{thm:Woodbury}
Let $k_0 \in \mathcal{K}_{\mathrm{wa}}$ and let $S =
\mathrm{span}\{w_1,\dots,w_r \}$ ($r=1$ or $r=2$) be the
(finite-dimensional) subspace of $C(\partial D_0)$ spanned by the
elements $w_j = \exp (i\alpha_{n_j}(k) x+ i\beta_{n_j}(k) y)$, with
$n_j$ ($j=1,\dots,r$) defined as in Remark~\ref{rem:WA_def}. If the
operator $A_{k_0}:C(\partial D_0) \to C(\partial D_0)$
(equation~\eqref{opA} with $k=k_0$) is invertible, then there exists
$\delta>0$ such that for $\abs{k-k_0}<\delta$ we have:
\begin{enumerate}
\item\label{thm:item1} The operator $A_{k}$ is invertible and the
  inverse $A_{k}^{-1}$ is a continuous function of $k$ for
  $\abs{k-k_0}<\delta$.
\item\label{thm:item2} The restriction $\tilde{T}_k =
  \left. \tilde{R}_k A_k^{-1}\right |_S$ of the composite operator
  $\tilde{R}_k A_k^{-1}$ to the subspace $S$ maps $S$ into itself
  bijectively and bicontinuously, and the inverse $\tilde{T}_k^{-1}$
  is a continuous function of $k$ for $\abs{k-k_0}<\delta$.
\item\label{thm:item3} Let
  $\boldsymbol{\beta(k)}=(\beta_{n_1}(k),\dots,\beta_{n_r}(k))$ and
  let $0<\abs{k-k_0}<\delta$. Then, the solution $\psi$ of
  equation~\eqref{intEqGj} (or,
  equivalently,~\eqref{operatorEquation}) is given by $\psi=\psi_{k}$,
  where
\begin{equation}
\label{eq:IESol}
\psi_{k}=-A_k^{-1} \left( \mathbb{I} - (D_{\boldsymbol{\beta(k)}}+\tilde{T}_{k})^{-1} \tilde{R}_{k}A_k^{-1} \right)u^{\mathrm{inc}}_k.
\end{equation}
Here the operator $D_{\boldsymbol{\beta(k)}}: S\to S$ is defined by
\begin{equation}\label{D_beta}
D_{\boldsymbol{\beta(k)}}w_j = \beta_{n_j} w_j.
\end{equation}
\end{enumerate}
Finally, the operator on the right hand side of
equation~\eqref{eq:IESol}, and therefore $\psi_k$ itself, are well
defined for $\abs{k-k_0}<\delta$, and, in particular, at $k=k_0$.  The
correspondence $k \to \psi_k$ maps $k\in (k_0-\delta, k_0+\delta)$
continuously into $C(\partial D_0)$. In particular, the
solutions~\eqref{eq:IESol} tend uniformly to $\psi_{k_0}$ as $k\to
k_0$.
\end{theorem}
\begin{proof}
  The existence of a certain $\delta_1>0$ such that $A_k$ admits an
  inverse $A_k^{-1}$ for $|k-k_0|<\delta_1$, as well as the continuity
  of $A_k^{-1}$ for $0<|k-k_0|<\delta_1$, follows easily from the
  invertibility of $A_{k_0}$---by means of a simple perturbative
  argument based on use of a Neumann series and the continuity of
  $A_k$ with respect to $k$ (Lemma~\ref{lem:contA_k}). Noting that the finite
  rank operators $\tilde{R}_k$ vary continuously, further, a similar
  perturbative argument allows one to deduce the invertibility of
  $\tilde{R}_kA_k^{-1}$ and the continuity of its inverse for
  $|k-k_0|<\delta_1$ (perhaps reducing the value of $\delta_1$, if
  necessary) from the invertibility of $\tilde{R}_{k_0}
  A_{k_0}^{-1}$--which is itself established in
  Appendix~\ref{app:invertibility} under the hypothesis of the present
  theorem. We have thus established that, for some $\delta_1>0$,
  Points~\ref{thm:item1} and~\ref{thm:item2} hold for all $k$
  satisfying $|k-k_0|<\delta_1$.

  Point~\ref{thm:item3}, in turn, follows by relating the inversion
  formula~\eqref{eq:WoodburyFormula} to
  equation~\eqref{operatorEquation}. Indeed, the invertibility of
  $\tilde{T}_k=\tilde{R}_k A_k^{-1}$ (Point \ref{thm:item2}) implies
  the invertibility of $D_{\boldsymbol{\beta}(k)}+\tilde{T}_k$
  provided $\abs{\beta_{n_j}}$ is sufficiently small for all
  $j$---which is certainly guaranteed provided $\abs{k-k_0}<\delta_2$
  for a sufficiently small value of $\delta_2>0$. Thus, identifying
  $D_{\boldsymbol{b}}$ and $T_{\boldsymbol{1}}$ (in
  Lemma~\ref{lem:Woodbury}) with $D_{\boldsymbol{\beta}(k)}$ and
  $\tilde{T}_k$ respectively, the
  $(D_{\boldsymbol{b}}+T_{\boldsymbol{1}})$ invertibility assumption
  in Lemma~\ref{lem:Woodbury} is satisfied for $\abs{k-k_0}<\delta_2$,
  and, therefore, equation~\eqref{eq:IESol} is obtained from
  equation~\eqref{eq:WoodburyFormula}, as desired.

  To complete the proof of the theorem we note that since
  $\tilde{T}_{k_0}= \tilde{R}_{k_0} A_{k_0}^{-1}:S \to S$ is
  invertible (Appendix~\ref{app:invertibility}), since the image of
  $\tilde{R}_{k_0} A_{k_0}^{-1}$ is contained in $S$, and since for
  $k=k_0$ we have $D_{\boldsymbol{\beta}(k)}=0$ (and thus
  $D_{\boldsymbol{\beta}(k_0)}+\tilde{T}_{k_0}=\tilde{T}_{k_0}$), it
  follows that the right hand side of equation~\eqref{eq:IESol} is
  also defined for $k=k_0$.  The uniform convergence of $\psi_k$ to
  $\psi_{k_0}$ as $k\to k_0$ is established by relying on the
  $k$-continuity at $k=k_0$ (established in Points~\eqref{thm:item1}
  and~\eqref{thm:item2}) of the operators in equation~\eqref{eq:IESol}
  together with the smoothness of the incident field as a function of
  $k$.
\end{proof}

The previous theorem relies on the invertibility of the operator
$A_{k_0}$. Unfortunately the study of solution uniqueness for the
operator $A_{k_0}$ for $k_0\in \mathcal{K}_\mathrm{wa}$ presents
difficulties: as pointed out in Remark~\ref{rem:invert_Ak}
(Appendix~\ref{app:Rayleigh}), the potential
\begin{equation}
\label{eq:potVk_0}
v_{k_0}(r)=\int \limits_{\partial D_0} \left(\partialDerivative{G_{j,k_0}^q}{\nu(r^\prime)}-i \gamma G_{j,k_0}^q\right) (r-r^\prime) \psi(r^\prime) dS(r^\prime)-\frac{i}{2L} \sum \limits_{n \not \in \mathcal{N}_{\mathrm{wa}}}  \sigma_n(k_0) \frac{I_{n,k_0}^+\left[\psi\right]}{\beta_n(k_0)} e^{i\alpha_n(k_0) x+i\beta_n(k_0) y}
\end{equation}
associated with $A_{k_0}$ is not necessarily a radiating solution and,
therefore, classical arguments based on uniqueness of radiating
solution for the associated PDE problem are not immediately applicable
in this context. A detailed consideration of this problem is out of
the scope of the present paper and is left for future work.
Throughout this article, however, it is assumed that, as verified
numerically (Section~\ref{sec:singularValues}), the operator $A_{k_0}$
is indeed invertible and, therefore, the conclusions of
Theorem~\ref{thm:Woodbury} hold.

\subsection{Solutions to the PDE problem~\eqref{PDE}--\eqref{TE} at
  and around Wood frequencies}
\label{theo_disc}

This section provides a brief discussion of equation~\eqref{eq:IESol},
and it uses the solutions to that equation (which, from
Theorem~\ref{thm:Woodbury}, are well defined for $k$ in the
neighborhood $U_{\mathrm{wa}} = \{k\in\mathbb{R}: |k- k_0|<\delta\}$
of a given Wood frequency $k_0\in\mathcal{K}_\mathrm{wa}$) to
construct solutions of the PDE~\eqref{PDE} for frequencies $k\in
U_{\mathrm{wa}}$--- including, in particular, the Wood frequency
$k=k_0$.

To do this we first note that the quantity $\psi_k$ in
equation~\eqref{eq:IESol} is the sum of the two terms
\begin{equation}
\label{eq:terms}
\psi_k^{(1)}=-A_k^{-1} u^{\mathrm{inc}}_k \quad \mbox{and}\quad \psi_k^{(2)}=A_k^{-1} \left( D_{\boldsymbol{\beta}(k)}+\tilde{T}_k\right)^{-1} \tilde{R}_k A^{-1}_k u^{\mathrm{inc}}_k,
\end{equation}
both of which involve the inverse $A_k^{-1}$ of the operator $A_k$. (A
theoretical discussion concerning the invertibility of the operator
$A_k$ is given at the end of Section~\ref{sec:Woodbury}, while
practical matters concerning actual inversion of $A_k$ are discussed
in Section~\ref{sec:numericalAlgorithm}.)

The term $\psi_k^{(1)}$ is obtained by a direct application of the
inverse operator $A_k^{-1}$. The evaluation of $\psi_k^{(2)}$ can be
viewed as a three-step process: (i)~Evaluation of $R_kA_k^{-1}
u_k^{\mathrm{inc}}$, (ii)~Application of the inverse
$(D_{\boldsymbol{\beta}(k)}+\tilde{T}_k)^{-1}$, and, finally,
(iii)~Application of $A_k^{-1}$. In view of
equation~\eqref{RDirichlet2} for point~(i) we have
\begin{equation}\label{tilde_R}
  \tilde{R}_k A^{-1}_k u^{\mathrm{inc}}_k = \sum \limits_{j=1}^r c_j w_j=\sum \limits_{j=1}^r c_j e^{i\alpha_{n_j}(k)x+i\beta_{n_j}(k) y},
\end{equation}
where $c_j=c_j(k) =-\frac{i}{2L} \sigma_{n_j}(k) I_{n_j,k}^+
\left[A_{k}^{-1} u^{\mathrm{inc}}_{k}\right]$. The inverse of the
finite-dimensional operator $D_{\boldsymbol{\beta}(k)}+\tilde{T}$
mentioned in point~(ii), on the other hand, can easily be applied to
$\tilde{R}_k A^{-1}_k u^{\mathrm{inc}}_k$ by first obtaining the
matrices $\mathcal{D}_{\boldsymbol{\beta}(k)}$ and $\mathcal{T}$ of
the operators $\mathcal{D}_{i,j}$ and $\mathcal{T}_{i,j}$ in the basis
$\left\{w_j=e^{i\alpha_{n_j}(k)x+i\beta_{n_j}(k) y}:j=1,\dots,r \right\}$
of the space $S$. We obtain
\begin{equation}
\label{eq:matrices}
  \mathcal{D}_{i,j}=\delta_i^j \beta_{n_j} \quad\mbox{and}\quad \mathcal{T}_{i,j}= -\frac{i}{2L} \sigma_{n_i}(k) I_{n_i,k}^+\left[A_{k}^{-1} w_j\right] \quad i,j=1,\dots,r.
\end{equation}
Thus, letting $c(k)=(c_1(k),\dots,c_r(k))^t\in \CC^r$ and calling
$d(k)=(d_1(k),\dots,d_r(k))^t \in \CC^r$ the solution of the linear system
\begin{equation}
\label{eq:d_coeffs}
(\mathcal{D}_{\boldsymbol{\beta}(k)}+\mathcal{T})d(k) = c(k),
\end{equation}
a straightforward linear algebra argument yields
\begin{equation}
\label{eq:finite_lc}
\left(D_{\boldsymbol{\beta}(k)}+\tilde{T}_k\right)^{-1}R_kA_k^{-1} u_k^{\mathrm{inc}} =\sum \limits_{j=1}^r d_j(k) w_j=\sum \limits_{j=1}^r d_j e^{i\alpha_{n_j}(k)x+i\beta_{n_j}(k)y}.
\end{equation}
(Note that for each $k$ a unique solution $d(k)$
of~\eqref{eq:d_coeffs} indeed exists and varies continuously with $k$
for $|k-k_0|$ sufficiently small, since, as indicated in
Theorem~\ref{thm:Woodbury}, $D_{\boldsymbol{\beta}(k)}+\tilde{T}_k$ is
invertible in a neighborhood of the Wood frequency $k_0$.)  Upon
application of $A^{-1}_k$ (point (iii)) we thus obtain the relation
\begin{equation}\label{eq:IESolDj}
\psi_{k} = -A_{k}^{-1}\left(u^{\mathrm{inc}}_{k}\right)+\sum \limits_{j=1}^r d_j(k) A_{k}^{-1} w_j.
\end{equation}
In other words, for $k\in U_{\mathrm{wa}}$ the density $\psi_{k}$ can
be viewed as the solution of an operator equation involving the
operator $A_{k}$ which is then corrected by addition of finitely many
terms which involve the Rayleigh modes
$w_j=e^{i\alpha_{n_j}(k)x+i\beta_{n_j}(k)y}$.

Once $\psi_{k}$ has been obtained as indicated above, the
corresponding radiating solution $u^{\mathrm{s}}_k$ of the PDE
problem~\eqref{PDE}--\eqref{TE} can be produced for $k\in
U_{\mathrm{wa}}$. While for $k\in U_{\mathrm{wa}}\setminus\{k_0\}$
(and, indeed, for all $k\not\in\mathcal{K}_\mathrm{wa}$) a direct
substitution $\psi=\psi_k$ in equations~\eqref{eq:potRepGj+}
and~\eqref{eq:potRepGj-} yields the desired PDE solution
$u^{\mathrm{s}}_k$, for all $k\in U_{\mathrm{wa}}$ we proceed
differently: as shown in what follows, for such $k$ values the
singular (or nearly singular) quotients
\begin{equation}\label{eq:sing_term}
  \frac{I_{n_j,k}^\pm[\psi_k]}{\beta_{n_j}(k)} \quad (n_j \in \mathcal{N}_{\mathrm{wa}}(k_0))
\end{equation}
in those two equations are re-expressed in terms of quantities that
are well-behaved at and around $k_0$, namely, the coefficients
$d_j(k)$ in equation~\eqref{eq:d_coeffs} and the functionals
$\tilde{I}_{n,k}$ defined in equation~\eqref{eq:functionalITilde}
below.

To do this let us consider the ``plus'' quotients first. Noting that
$\psi_k$ is a solution of~\eqref{operatorEquation} we obtain the
expression
\begin{equation}
R_k \psi_k =-u^{\mathrm{inc}}_k-A_k \psi_k
\end{equation}
which, using equation~\eqref{eq:IESolDj}, becomes
\begin{equation}
\label{eq:R_kPsi_k}
R_k\psi_k=\sum \limits_{j=1}^r d_j(k) e^{i\alpha_{n_j}(k)x+i\beta_{n_j}(k)y}.
\end{equation}
In view of~\eqref{RDirichlet} it follows that for all $k \in
U_{\mathrm{wa}}$ the ``plus'' singular quotient in~\eqref{eq:sing_term}
can be expressed in the form
\begin{equation}
\label{lim_plus}
-\frac{i}{2L} \sigma_n(k) \frac{I_{n_j,k}^+\left[\psi_k\right]}{\beta_{n_j}(k)}=d_j(k) , \quad j=1,\dots,r
\end{equation}
in terms of the ``regular'' quantities $d_j = d_j(k)$. 

The ``minus'' quotients, in turn, can be expressed in terms of the
``plus'' quotients. Indeed, subtracting~\eqref{InMinus}
from~\eqref{InPlus} we obtain
\begin{equation}
  I_{n,k}^+[\psi]-I_{n,k}^-[\psi]= -2i \beta_n(k) \tilde{I}_{n,k}[\psi]
\end{equation}
where
\begin{equation}
\label{eq:functionalITilde}
\tilde{I}_{n,k}[\psi]=\int \limits_{\partial D_0} \left( \partialDerivative{}{\nu(r^\prime)} \left( e^{-i\alpha_n(k) x^\prime} y^\prime \mbox{sinc}(\beta_n(k) y^\prime) \right) -i\gamma e^{-i\alpha_n(k) x^\prime} y^\prime \mbox{sinc}(\beta_n(k) y^\prime) \right) \psi(r^\prime)dS(r^\prime),
\end{equation}
and where $\mbox{sinc}(t)=\frac{\sin(t)}{t}$. Thus, in view
of~\eqref{lim_plus} we obtain the expression
\begin{equation}\label{lim_minus}
\frac{i}{2L} \frac{I_{n_j,k}^- [\psi_{k}]}{\beta_{n_j}(k)}=-\frac{d_j(k)}{\sigma_{n_j}(k)} -\frac{1}{L} \tilde{I}_{n_j,k_0} [\psi_{k_0}], \quad j=1,\dots,r
\end{equation}
for the ``minus'' quotients in terms of regular quantities.

The scattered-field functions~\eqref{eq:potRepGj+}
and~\eqref{eq:potRepGj-} can now be evaluated in terms of the regular
expressions~\eqref{lim_plus} and~\eqref{lim_minus}. Indeed, using the
solution $d_j(k)$ of equation~\eqref{eq:d_coeffs} and the integers
$n_j$ ($j=1,\dots,r$) introduced in Remark~\ref{rem:WA_def} we obtain
\begin{equation}
\label{eq:potRepGj+WA}
\begin{split}
u^{\mathrm{s}}_{k}(r)=\int \limits_{\partial D_0} \left(\partialDerivative{G_{j,k}^q}{\nu(r^\prime)}-i \gamma G_{j,k}^q\right) (r-r^\prime) \psi_{k}(r^\prime) dS(r^\prime)&-\frac{i}{2L}\sum \limits_{n \not \in \mathcal{N}_{\mathrm{wa}}}  \sigma_n(k) \frac{I_{n,k}^+\left[\psi_{k}\right]}{\beta_n(k)} e^{i\alpha_n(k) x+i\beta_n(k) y} \\&+ \sum \limits_{j=1}^r d_j(k) e^{i\alpha_{n_j}(k) x+i\beta_{n_j}(k)y}
\end{split}
\end{equation}
in $\Omega^+=\{y\geq M^--h\}\setminus D$ (see Figure~\ref{problemGeometry}); and
\begin{equation}
\label{eq:potRepGj-WA}
u^{\mathrm{s}}_{k}(r)=\frac{i}{2L}\sum \limits_{n\not \in \mathcal{N}_{\mathrm{wa}}}  \frac{I_{n,k}^-\left[\psi_k\right]}{\beta_n(k)} e^{i\alpha_n(k) x -i\beta_n(k) y} + \sum \limits_{j=1}^r \left(d_j(k) -\frac{1}{L} \tilde{I}_{n_j,k}[\psi_{k}]\right) e^{i\alpha_{n_j}(k) x-i\beta_{n_j}(k)y}
\end{equation}
in $\Omega^-=\{y<M^--h\}$.  Note that~\eqref{eq:potRepGj+WA}
and~\eqref{eq:potRepGj-WA} are defined for all values of $k \in
U_{\mathrm{wa}}$---including $k=k_0$.

To conclude this Section we establish that, as claimed, the
Wood-anomaly potential $u^{\mathrm{s}}_{k_0}$ given by
equations~\eqref{eq:potRepGj+WA} and~\eqref{eq:potRepGj-WA} is a
radiating solution of the problem~\eqref{PDE}--\eqref{TE}. To do this
it suffices to show that
\begin{enumerate}
\item\label{pt1} $u^{\mathrm{s}}_{k_0}$ is a solution of the Helmholtz equation
  in the interior of the regions $\Omega^+$ and $\Omega^-$;
\item\label{pt2} The right-hand expressions in~\eqref{eq:potRepGj+WA}
  and~\eqref{eq:potRepGj-WA} agree in the region $M^- - h <y <M^-$;
\item\label{pt3} $u^{\mathrm{s}}_{k_0}$ satisfies the condition of
  radiation at infinity in both $\Omega^+$ and $\Omega^-$ (see
  Remark~\ref{rem:Radiating});
\item\label{pt4} $u^{\mathrm{s}}_{k_0}$ verifies the boundary
  condition~\eqref{TE} or, equivalently, the boundary
  condition~\eqref{TE0}.
\end{enumerate}
The validity of point~\ref{pt1} follows directly by inspection of
equations~\eqref{eq:potRepGj+WA} and~\eqref{eq:potRepGj-WA}.  To
establish point~\ref{pt2}, in turn, we show that, in the region $M^- -
h <y <M^-$ the Rayleigh expansion of the right-hand
in~\eqref{eq:potRepGj+WA} coincides with the Rayleigh
series~\eqref{eq:potRepGj-WA}. To produce the Rayleigh series
of~\eqref{eq:potRepGj+WA} we use equation~\eqref{eq:potVk_0} with
$\psi=\psi_{k_0}$, which results in an expression of the form
\begin{equation}\label{eq:relationUandV}
  u^{\mathrm{s}}_{k_0}(x,y)=v_{k_0}(x,y) + \sum \limits_{j=1}^r d_j(k_0) e^{i\alpha_{n_j}(k_0) x};
\end{equation}
Point~\ref{pt2} now follows directly by considering the Rayleigh
expansion of $v_{k_0}$ (equation~\eqref{RayleighBelow} in
Appendix~\ref{app:Rayleigh}).  Point~\ref{pt3} follows directly by
inspection of equations~\eqref{eq:potRepGj-WA}
and~\eqref{eq:relationUandV}, using, for the latter, the appendix
equation~\eqref{RayleighAbove}.  Point~\ref{pt4}, finally, results by
evaluation of the potential $u^{\mathrm{s}}_{k_0}$
in~\eqref{eq:relationUandV} for points $(x,y)\in\partial D_0$
(cf. equation~\eqref{TE0} and associated text). The boundary values of
the potential $v_{k_0}$ in equation~\eqref{eq:relationUandV} can be
obtained by evaluating~\eqref{eq:potVk_0} (with $\psi=\psi_{k_0}$) on
$\partial D_0$; the result is~\eqref{opA} (with $\psi=\psi_{k_0}$) or,
in other words $v_{k_0} \big |_{\partial D_0}=A_{k_0}\psi_{k_0}$.  But
this quantity can be obtained from Theorem~\ref{thm:Woodbury} applied
at $k=k_0$ or, equivalently, from equation~\eqref{eq:IESolDj}. But,
then, equation~\eqref{eq:relationUandV} tells us that
$u^{\mathrm{s}}_{k_0}$ verifies the boundary condition~\eqref{TE}, as
desired. Having established that points~\ref{pt1} through~\ref{pt4}
hold, it follows that $u^{\mathrm{s}}_{k_0}$ is a radiating solution
of the PDE problem posed by equations~\eqref{PDE} and~\eqref{TE} at
the Wood frequency $k_0$, as desired.

\section{Numerical algorithm\label{sec:numericalAlgorithm}}
Section~\ref{sec:formul} presents an integral-equation framework for
evaluation of scattering solutions of the PDE
problem~\eqref{PDE}--\eqref{TE} for wavenumbers $k$ at and around a
given Wood wavenumber $k_0$. For values of $k$ away from all Wood
wavenumbers, in turn, the proposed algorithm resorts to use of the
windowed Green function approach introduced
in~\cite{BrunoShipman,BrunoDelourme,Monro}. In fact, following these
references, our algorithm applies the windowing approach in
conjunction with the shifted Green function method at all
wavenumbers. This section presents a numerical discretization of the
resulting continuous formulations.  We consider the near Wood anomaly
case first and we then succinctly describe the modifications that are
necessary to obtain an algorithm valid for all frequencies.

For values of $k$ at and around a given Wood frequency $k_0\in
\mathcal{K}_\mathrm{wa}$, the proposed algorithm for evaluation of the
density $\psi_k$ (that is used in
equation~\eqref{eq:potRepGj+WA}--\eqref{eq:potRepGj-WA}) results as a
discrete analog of the strategy embodied in
equation~\eqref{eq:IESol}---or, equivalently (but in a form more
closely related to the actual implementation)
equation~\eqref{eq:IESolDj}. The numerical integrations mentioned in
what follows can be effected by means of any numerical integration
method applicable to the kinds of smooth and logarithmic-singular
integrands associated with the problems under consideration. The
integration algorithms used in our implementation were derived in a
direct manner from the high-order pointwise-discretization (Nystr\"om)
methods described in~\cite[Sec. 3.5]{kressInverse}; in particular,
this presentation assumes a discretization of the boundary of the
scatterer $D_0$ by means of a given Nystr\"om mesh such as those
considered in~\cite{kressInverse}. Naturally, only one period of the
scattering scattering boundaries needs to be discretized; we assume a
number $n_i$ of Nystr\"om nodes are used to discretize the boundaries
contained in the reference unit period. Finally, the exponentially
convergent infinite sum in equation~\eqref{opA} is truncated by
including all propagating modes as well as a number $N_{ev}$ of
evanescent modes to the right (resp. to the left) of the largest
(resp. smallest) element of the set $\mathcal{N}$, excepting all Wood
frequency modes.

Away from Wood frequencies the evaluation of the shifted
quasi-periodic Green function $G^q_{j,k}$ needed to produce the
operator $A_k$ and its inverse, which is required
in~\eqref{eq:IESolDj} (see equation~\eqref{opA}), can be additionally
accelerated by means of the windowing methodology introduced
in~\cite{BrunoShipman,BrunoDelourme}---which smoothly truncates the
infinite sum~\eqref{quasiPerGj} in such a way that $G^q_{j,k}$ is
approximated by the finite sum
\begin{equation}
\label{def:G^A}
 G^A_{j,k}(x,y)=\sum \limits_{n \in \ZZ} e^{-i\alpha n L}G_{j,k}\left( k u(x+n L,y)\right) W \left( \frac{x+nL}{A}\right),
\end{equation}
where $A$ is a positive real number, and where $W:\RR \to \RR$ (see
Figure~\ref{window}) is the smooth cut-off function given by
\begin{equation}
  W(t)=\left\{ \begin{array}{ll}  1 \quad & \mathrm{if} \, \abs{t} \leq c 
\\\exp(\frac{2e^{-1/\abs{t}}}{\abs{t}-1}) \quad &\mathrm{if} \, c\leq \abs{t} \leq 1
\\ 0 \quad &\mathrm{if} \, \abs{t}\geq 1\end{array}
\right. 
\end{equation}
with $0<c<1$.
\begin{figure}[t]
 \centering
    \includegraphics[scale=0.6]{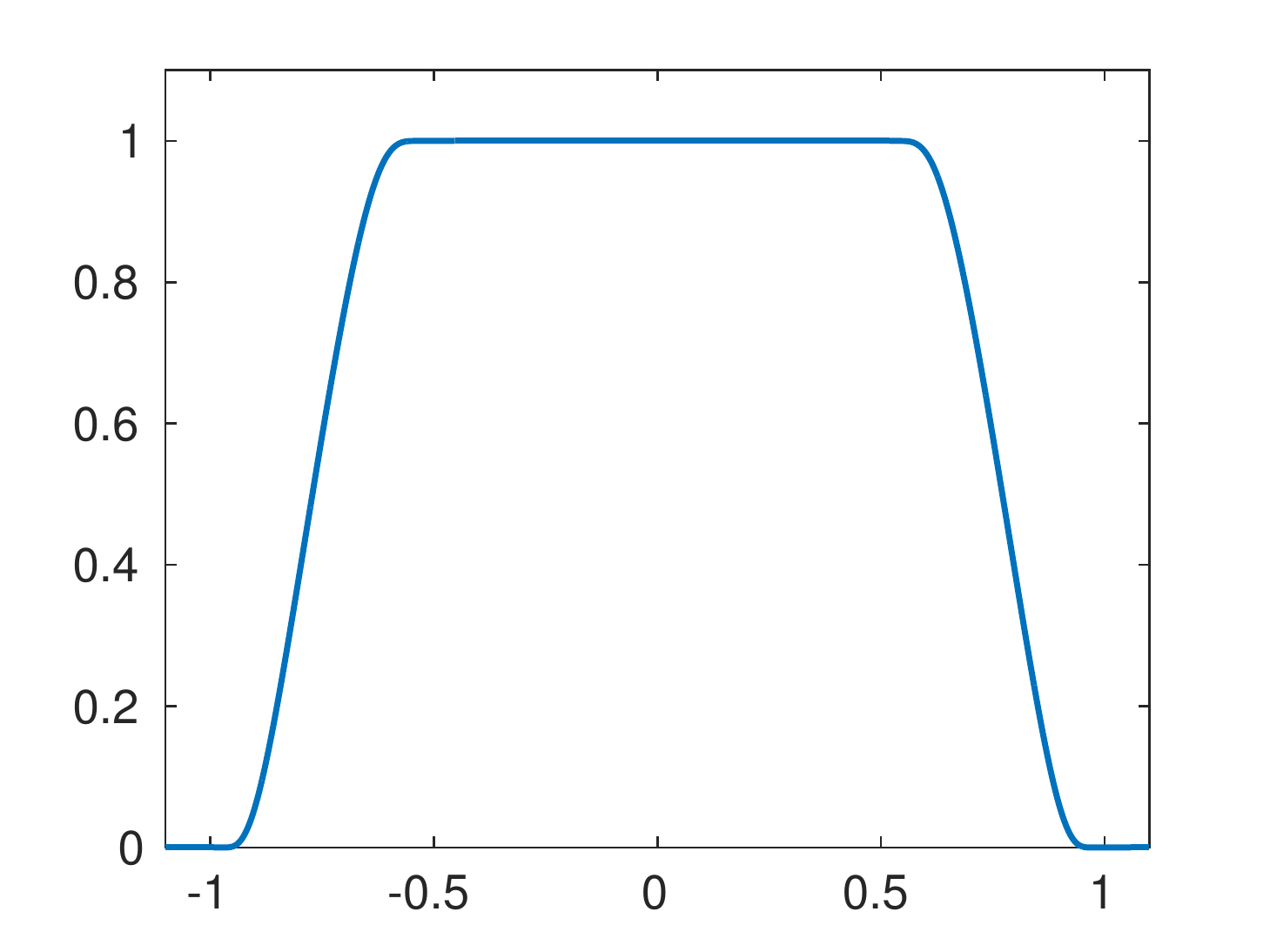}
\caption{Windowing function $W$ for $c=0.5$. \label{window}}
\end{figure}
As shown in~\cite{BrunoShipman} the smooth-windowing methodology
converges super-algebraically fast away from Wood anomalies as $A\to
\infty$, and thus this approach improves upon the shifted Green
function convergence---at least away from Wood
frequencies. Reference~\cite{BrunoShipman} establishes super-algebraic
convergence to the corresponding quasi-periodic Green function in the
context of the problem of scattering by bi-periodic structures in
three-dimensional space. Similar ideas can be applied in the
two-dimensional problem under consideration. In fact the proof is
simpler in the present case---which does not require summation of
doubly-infinite sums.

We can now consider the problem of evaluation of the operator $A_k$,
whose inverse appears on the right-hand side
of~\eqref{eq:IESolDj}. Clearly, using~\eqref{opA} and the values of a
given continuous function $\psi$ on the Nystr\"om mesh, we can obtain
the numerical values of $A_k\psi$ on the same Nystr\"om mesh by
numerical integration of the integral terms containing the Green
function as well as those associated with the functionals $I^+_{n,k}$,
followed by summation of the infinite sum over
$n\not\in\mathcal{N}_\mathrm{wa}$ (whose general term decays
exponentially, in accordance with~\eqref{contA_k_eqn}). Clearly, such
an algorithm can be implemented in terms of a matrix-vector product
for the vector which contains the discrete values of the function
$\psi$. The inverse of the corresponding matrix provides the necessary
discrete approximation of the operator $A_k^{-1}$. (In practice the
action of the numerical inverse on a given vector can be obtained
either by means of an iterative linear-algebra solver or, as in the
approach followed in the present two-dimensional context, directly by
means of Gaussian elimination.)

In order to complete the evaluation of the solution $\psi_k$
in~\eqref{eq:IESolDj} it is necessary to produce the the coefficients
$d_j(k)$. But these values can be obtained by solving the $r\times r$
linear system~\eqref{eq:d_coeffs} (cf. Remark~\ref{rem:WA_def}). The
necessary elements in this matrix equation can be produced as follows:
the diagonal matrix $\mathcal{D}_{\boldsymbol{\beta}(k)}$ results from
simple algebra, and the matrix $\mathcal{T}$
(equation~\eqref{eq:matrices}) and right-hand side $c(k)$ (right below
equation~\eqref{tilde_R}) can be produced through respective
applications of the aforementioned matrix form of the operator
$A_k^{-1}$ followed by numerical integration and simple algebra.

The case in which $k$ is away from Wood anomalies, finally, can be
treated by means of a slightly modified version of the ``Wood and
near-Wood'' strategy described above---since, as it is readily
checked, the operator $A_k$ in equation~\eqref{opA} and the mixed
integral/Rayleigh-series operator in equation~\eqref{intEqGj} differ
only by the finite sum of terms with
$n\in\mathcal{N}_{\mathrm{wa}}$. Thus, for a configuration away from
Wood anomalies it is only necessary to incorporate those terms which
were excluded to obtain $A_k$ in the near Wood-anomaly case. (Note,
however, that, away from Wood frequencies on may select $j=0$
in~\eqref{intEqGj}, in which case the sum on the left-hand side of
this equation vanishes, and a classical formulation is recovered.) In
any case, the resulting discrete formulations can be inverted by means
of an iterative solver or, for sufficiently small discretizations, by
means of Gaussian elimination.

\section{Numerical results\label{sec:numericalResults}}

This section presents numerical tests and examples that demonstrate
the character of the scattering solvers introduced in the previous
section, with an emphasis on the impact of the Wood-anomaly methods
described in Section~\ref{sec:Woodbury}.  Detailed numerical results
presented in what follows concern, in particular, the invertibility of
the operator $A_k$ around Wood frequencies (see
Theorem~\ref{thm:Woodbury} and discussion immediately thereafter), as
well as the conditioning, accuracy and computing costs associated with
the proposed methodology for ranges of frequencies which, once again,
include Wood anomalies.  All solutions presented in this section were
produced by means of a Fortran implementation of the algorithms
described in Section~\ref{sec:numericalAlgorithm}, together with the
LAPACK Gaussian elimination routines, on a single core of an Intel
i7-4600M processor.

\begin{figure}[H]
 \centering
 \includegraphics[scale=0.4]{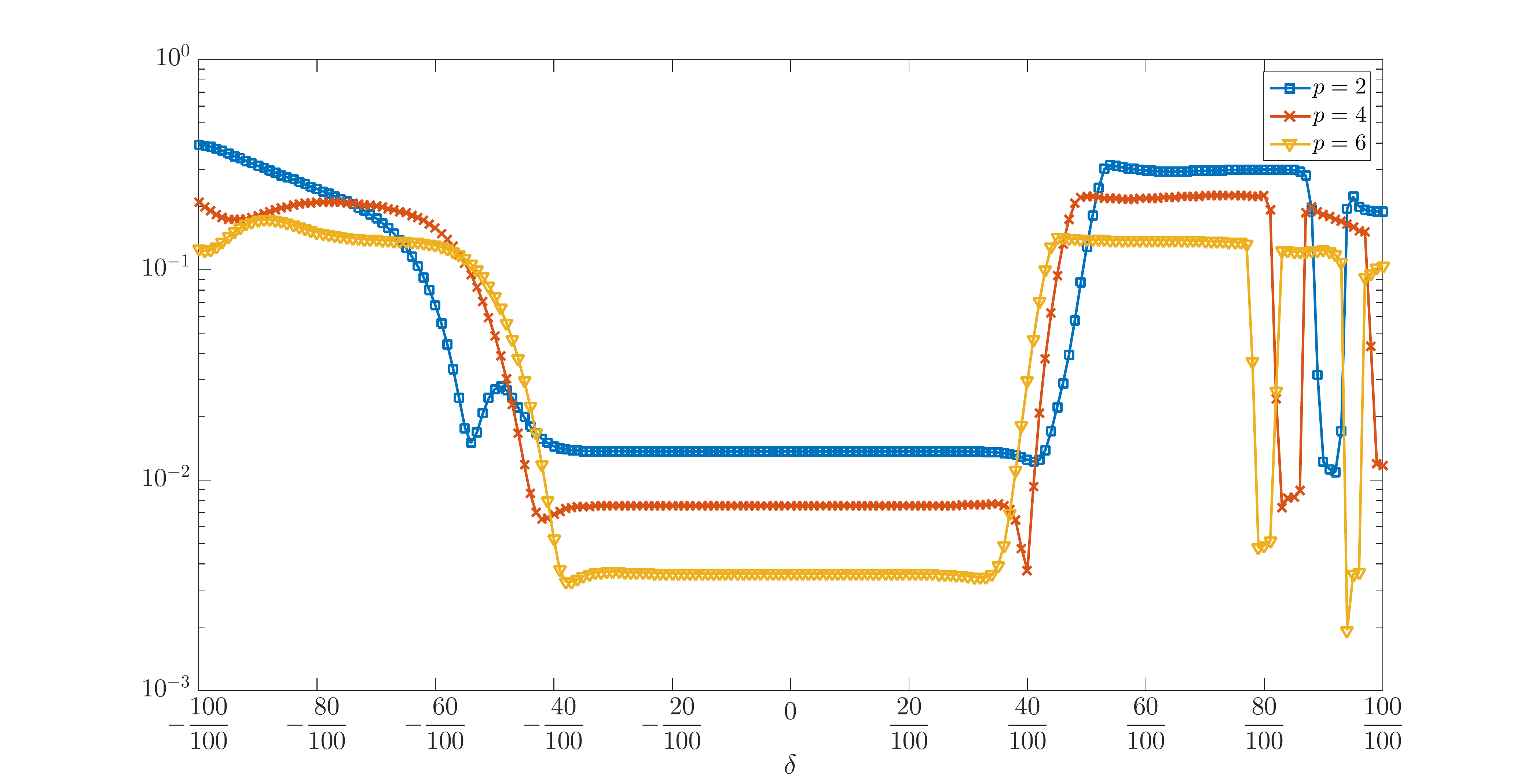}
 \caption{Smallest singular values of the discrete approximation
   $\tilde A_k$ of the operator $A_k$, for the scattering structure
   described in the text, as a function of $\delta$ ($k = k_0 +
   \delta^8$), around the Wood frequencies $k_0=p
   \frac{2\pi}{L(1+\sin(\theta))}$ with $p=2,4,6$. (For reference, the
   maximum condition numbers $\kappa =
   \sigma_\mathrm{max}/\sigma_\mathrm{min}$ for the cases $p=2$, $4$
   and $6$ are approximately $280$, $1508$ and $2401$, respectively.)
   Note that, in particular, the smallest singular values for
   $\delta^8 = 10^{-16}$ and $\delta =0$
   ($k_0\in\mathcal{K}_{\mathrm{wa}}$) are included in this graph.}
    \label{fig:results_1}
\end{figure}
\subsection{Invertibility and conditioning at and around Wood
  frequencies}\label{sec:singularValues}

In order to provide numerical evidence of the invertibility of the
operator $A_k$ (equation~\eqref{opA}) for $k$ in a neighborhood of
$k_0 \in\mathcal{K}_{\mathrm{wa}}$ we consider the maximum and minimum
singular values for a discretization of the operator $A_k$ based on
$n_i=64$ Nystr\"om nodes (see Section~\ref{sec:numericalAlgorithm}),
with varying values of both, the window-size parameter $A$ and the
number $j$ of shifts, and with $N_{ev} = 20$ (see
Section~\ref{sec:numericalAlgorithm}). We have found that, in all
cases the maximum singular values of $A_k$ do not grow as $k$
approaches $k_0$, and that the minimum singular values are all bounded
away from zero. For definiteness we present examples for the fairly
generic test geometry depicted in Figure~\ref{problemGeometry} with
$L=4$ and $\theta=0$; similar results were obtained for periodic
arrays of cylinders of various cross sections and for other numbers
$n_i$ of discretization points.

Figure~\ref{fig:results_1} displays the smallest singular values of
the discrete approximation $\tilde A_k$ of the operator $A_k$ in
equation~\eqref{opA} obtained using $A=1600\cdot L$ together with a
$64$-point discretization of the scatterer $D_0$ depicted in
Figure~\ref{problemGeometry}. In order to provide close refinements
near several Wood anomalies we use the parameters $\delta$ defined by
$k = k_0 + \mathrm{sign}(\delta)\cdot\delta^8$ for $k_0=p
\frac{2\pi}{L(1+\sin(\theta))}\in\mathcal{K}_{\mathrm{wa}}$ with
$p=2,4,6$. (The $\delta$ parametrization of the frequency domain is
used to achieve a fine resolution near the Wood anomaly values
corresponding to $\delta=0$ for each one of the integers $p=2,4,6$.)
Visually indistinguishable curves for both smallest and largest
singular values are obtained using a 128-point discretization of $D_0$
(while fixing the value $A=1600 L$), which suggests that the
approximate singular values provide close approximations of the
corresponding smallest and largest singular values of the continuous
operator $A_k$. Clearly, the smallest singular value
$\sigma_\mathrm{min}$ remains bounded in the limit as $\delta \to 0$.

Figure~\ref{engy_bal} presents the energy balance errors (given by the
the left-hand side of equation~\eqref{engy_bal_eqn} for efficiencies
values $e^\pm_n$ computed numerically with $n_i=64$) as a function of
$\delta$, with window-size $A=LN_{per}$ and for two values of the
shift-parameter $j$ (equations~\eqref{Gj_def} and~\eqref{quasiPerGj});
a shifted Green function with a number $j=5$ of shifts and spacing
$h=3.5$ was used in all cases. Studies based on use of finer
integration meshes and larger window sizes suggest that the actual
solution errors are well described by the energy-balance curves
presented in this figure. This figure demonstrates clearly the
beneficial effect induced by the presence of the Green-function
shifts.
\begin{figure}[ht]
\begin{center}
    \includegraphics[scale=0.4]{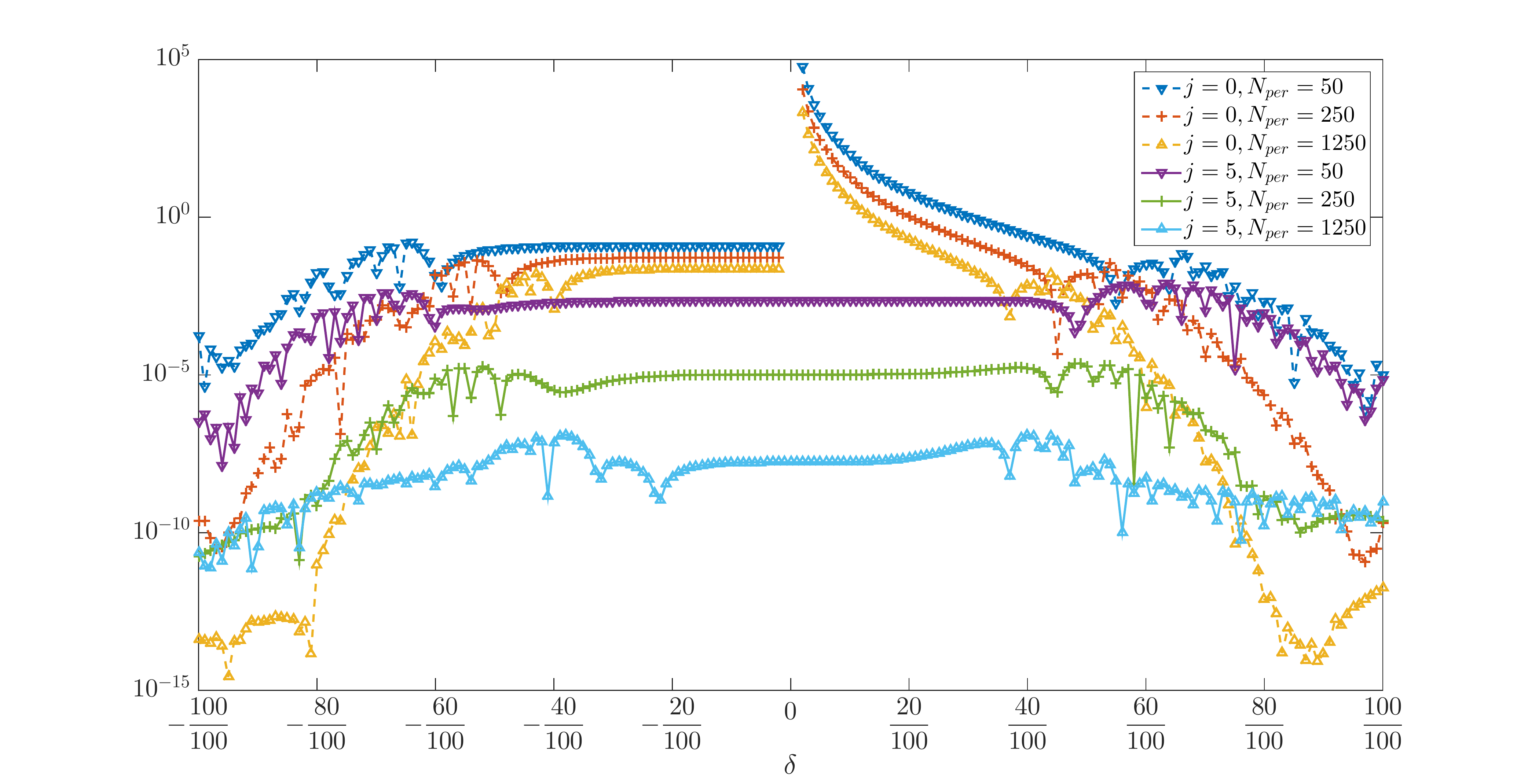}
\end{center}
\caption{Energy balance error, as a function of $\delta$
  ($k=k_0+\delta^8$ with $k_0=\frac{2\pi}{L(1+\sin(\theta))}$),
  resulting from use of the methodology introduced in
  Section~\ref{sec:Woodbury} with $j=0$ (un-shifted Windowed Green
  function) and $j=5$ (five-shift Windowed Green function) for a
  configuration with $L=4$ and $\theta=0$ (normal incidence). The
  beneficial effect induced by the use of the shifted Green-function
  and Sherman-Morrison inversion (cf. Theorem~\ref{thm:Woodbury}) can
  be clearly appreciated. \label{engy_bal}}
\end{figure}


\subsection{Computing cost}\label{sec:convergence}
The Green-function representations and solvers presented in this paper
give rise to fast numerical algorithms: like the rough-surface
solvers~\cite{BrunoDelourme}, the present methods for periodic arrays
of scatterers enable evaluation of highly-accurate scattering
solutions, with frequencies away, at and around Wood anomalies, in
computing times of the order of hundreds of milliseconds. In
particular, the approach is highly competitive with other available
solvers even for configurations away from Wood
anomalies~\cite{BarnettGreengard1,BrunoHaslam1}. The method can be
additionally accelerated by means of FFT-based
approaches~\cite{BrunoKunyansky,BrunoShipman}; use of such
acceleration methods in conjunction with shifted Green functions will
be described elsewhere.

Sections~\ref{sec:costAway}, \ref{sec:costAt} and \ref{sec:costAround}
present computing times required by the present solver for problems of
scattering by periodic arrays (period $L=2\pi$) of circular cylinders.
Examples for arrays of cylinders of radii $R=0.05 L$, $R=0.1 L$ and
$R=0.25 L$, and configurations far from Wood anomalies, at Wood
anomalies and around Wood anomalies are considered. Incidence angles
in Littrow mount of order $\ell=-1$ (for which the scattered
plane-wave of order $\ell=-1$ propagates in the backscattering
direction) were used; such a configuration is obtained provided the
triplet $(k,L,\theta)$ verifies the relation $kL\sin(\theta)=\pi$. As
noted in the previous section, Wood frequencies can be obtained by
enforcing the relation $k_0=\frac{2\pi}{L(1+\sin(\theta))} n$ for some
positive integer $n$; it can be easily checked that under the
Littrow-mount assumption, the Wood-frequency condition reduces to $k_0
L =(2n-1) \pi$ for some positive integer $n$. In what follows we
consider the particular case $n=2$, which yields the Wood frequency
$k_0 =1.5 \cdot \frac{2\pi}{L}$ (for which
$\mathcal{N}_\mathrm{wa}=\{1,-2\}$; see Remark~\ref{rem:WA_def}). In
particular, $k=1\cdot\frac{2\pi}{L}$ is not a Wood frequency while the
frequency $k=1.49\cdot\frac{2\pi}{L}$ is close to the Wood frequency
$k_0$.

Throughout this section the computing times cited sufficed to produce
full scattering solutions with an energy balance error of the order of
$10^{-8}$.
\subsubsection{Computing cost I: wave-numbers away from Wood anomalies}
\label{sec:costAway}
Tables~\ref{tab:nonAnomalousCircle0.05},~\ref{tab:nonAnomalousCircle0.1}
and~\ref{tab:nonAnomalousCircle0.25} present the computing times
required by the proposed algorithm to evaluate the scattered field
with an energy-balance error of the order of $10^{-8}$ for the
wavenumber $k=1\cdot L$ (away from Wood anomalies).

\begin{table}[H]
\centering
\caption{Cylinders of radius $R=0.05 L$ at the frequency $k=1\cdot \frac{2\pi}{L}$. \label{tab:nonAnomalousCircle0.05}}

\begin{tabular} { l l l l l l l l }
  \hline
  $j$ &0 & 1    & 2    & 3    & 4    \\ \hline
  $n_i$&18 & 16   & 18   & 18   & 18   \\
  $N_{per}$&100 & 30   & 45   & 24   & 22   \\
  $N_{ev}$&0 & 18   & 18   & 16   & 16   \\
  Time (s)&0.06 & 0.04 & 0.08 & 0.08 & 0.07 \\ \hline
\end{tabular}
\end{table}

\begin{table}[H]
\centering
\caption{Cylinders of radius $R=0.1 L$ at the frequency $k=1\cdot\frac{2\pi}{L}$.\label{tab:nonAnomalousCircle0.1}}
\begin{tabular} { l l l l l l l l }
  \hline
  \multicolumn{1}{l|}{$j$} & 0 &1    & 2    & 3    & 4    \\ \hline
  $n_i$                   & 18 & 16   & 16   & 16   & 18   \\
  $N_{per}$             & 125 & 38   & 38   & 37   & 36   \\
  $N_{ev}$              & 0 &20   & 20   & 20   & 20   \\
  Time (s)               & 0.1 &0.07 & 0.07 & 0.08 & 0.09 \\ \hline
\end{tabular}
\end{table}

\begin{table}[H]
  \caption{Cylinders of radius $R=0.25 L$ at the frequency $k=1 \cdot\frac{2\pi}{L}$. \label{tab:nonAnomalousCircle0.25}}
\centering
\begin{tabular} { l l l l l l l l }
\hline
$j$          &0& 1    & 2    & 3    & 4    \\ \hline
$n_i$       &20& 14   & 14   & 14   & 16   \\
$N_{per}$ &125& 66   & 75   & 73   & 58   \\
$N_{ev}$  &0& 10   & 10   & 10   & 10   \\
Time (s)   &0.13& 0.06 & 0.09 & 0.11 & 0.13 \\ \hline
\end{tabular}
\end{table}

\subsubsection{Computing cost II: a Wood anomaly frequency }\label{sec:costAt}
Tables~\ref{tab:nonAnomalousCircle0.25atWA},~\ref{tab:nonAnomalousCircle0.1atWA}
and~\ref{tab:nonAnomalousCircle0.25atWA} present results for the
structures considered in the previous section except for the
frequency, which is here taken to equal the Wood frequency $k=1.5\cdot
L$. The best computing costs are now $0.14$, $0.6$ and $2.6$
seconds. These times are thus somewhat higher than the corresponding
$0.04$, $0.07$ and $0.06$ times required for the non-Wood frequency
$k=1\cdot L$.
\begin{table}[H]
  \centering
\caption{Cylinders of radius $R=0.05 L$ at the Wood frequency $k=1.5\cdot \frac{2\pi}{L}$. \label{tab:nonAnomalousCircle0.05atWA}}
\begin{tabular}{ l l l l l l l l }
\hline
$j$        & 1    & 2    & 3    & 4    & 5    & 6    \\ \hline
$n_i$    & 18   & 20   & 16   & 18   & 18   & 18   \\
$N_{per}$   & 450  & 600  & 100  & 100   & 50   & 30   \\
$N_{ev}$    & 20   & 20   & 20   & 20   & 20   & 20   \\
Time (s) & 0.62 & 0.8 & 0.25 & 0.31 & 0.19 & 0.14 \\ \hline
\end{tabular}
\end{table}

\begin{table}[H]
\centering
\caption{Cylinders of radius $R=0.1 L$ at the Wood frequency $k=1.5\cdot \frac{2\pi}{L}$.\label{tab:nonAnomalousCircle0.1atWA}}

\begin{tabular}{ l l l l l l l l }
\hline
$j$       &  1    & 2    & 3    & 4    & 5   & 6   \\ \hline
$n_i$      & 18   & 18   & 16   & 18   & 18  & 18  \\
$N_{per}$  & 3500 & 5500 & 500  & 400  & 200 & 200 \\
$N_{ev}$  & 20   & 20   & 20   & 20   & 20  & 20  \\
Time (s) & 2.35 & 3.96 & 0.7 & 0.6 & 0.7 & 0.8 \\ \hline
\end{tabular}
\end{table}

\begin{table}[H]
\centering
\caption{Cylinders of radius $R=0.25 L$ at the Wood frequency $k=1.5\cdot \frac{2\pi}{L}$.\label{tab:nonAnomalousCircle0.25atWA}}

\begin{tabular} { l l l l l l l l }
\hline
$j$       &    1   & 2     & 3    & 4    & 5    & 6   \\ \hline
$n_i$     & 20    & 20    & 20   & 20   & 20   & 20  \\
$N_{per}$ & 38500 & 42000 & 2500 & 2000 & 1000 & 750 \\
$N_{ev}$   & 20    & 20    & 20   & 20   & 20   & 20  \\
Time (s) & 32.1  & 35.5  & 4.5  & 4.5  & 2.6  & 2.6 \\ \hline
\end{tabular}
\end{table}

\subsubsection{Computing cost III: frequencies near Wood
  anomalies} \label{sec:costAround}
Tables~\ref{tab:nonAnomalousCircle0.05aroundtWA},~\ref{tab:nonAnomalousCircle0.1aroundtWA}
and~\ref{tab:nonAnomalousCircle0.25aroundtWA} present results for the
structures considered in the previous two sections, but now for the
near-Wood frequency $k=1.49\cdot L$. The computing times are
comparable to Wood-frequency times presented in the previous section.
\begin{table}[H]
\centering
\caption{Cylinders of radius $R=0.05 L$ at the frequency $k=1.49\frac{2\pi}{L}$.\label{tab:nonAnomalousCircle0.05aroundtWA}}
\begin{tabular}{ l l l l l l l l }
\hline
$j$       &0 & 1    & 2   & 3    & 4   & 5    & 6    \\ \hline
$n_i$     &18 & 16   & 16  & 16   & 16  & 18   & 18   \\
$N_{per}$ &3700 & 850  & 650 & 200  & 200 & 100   & 75   \\
$N_{ev}$ & 0 & 20   & 20  & 20   & 20  & 20   & 20   \\
Time (s)&3.2 & 0.59 & 0.5 & 0.36 & 0.4 & 0.37 & 0.32 \\ \hline
\end{tabular}
\end{table}

\begin{table}[H]
\centering
\caption{Cylinders of radius $R=0.1 L$ at the frequency $k=1.49\frac{2\pi}{L}$. \label{tab:nonAnomalousCircle0.1aroundtWA}}
\begin{tabular}{ l l l l l l l l }
\hline
$j$     & 0  & 1    & 2    & 3    & 4    & 5    & 6   \\ \hline
$n_i$   & 18  & 18   & 18   & 18   & 18   & 18   & 18  \\
$N_{per}$& 3700 & 900  & 1000 & 350  & 320  & 175  & 100 \\
$N_{ev}$ & 0  & 20   & 20   & 20   & 20   & 20   & 20  \\
Time (s) & 3.2 & 0.76 & 0.91 & 0.46 & 0.49 & 0.47 & 0.4 \\ \hline
\end{tabular}
\end{table}

\begin{table}[H]
\centering
\caption{Cylinders of radius $R=0.25 L$  at the frequency $k=1.49\frac{2\pi}{L}$.\label{tab:nonAnomalousCircle0.25aroundtWA}}

\begin{tabular}{ l l l l l l l l }
\hline
$j$      & 0 & 1    & 2    & 3    & 4    & 5    & 6    \\ \hline
$n_i$    & 20 & 20   & 20   & 20   & 20   & 20   & 20   \\
$N_{per}$ &3700 & 1200 & 1000 & 600  & 600  & 400  & 380  \\
$N_{ev}$  &0 & 10   & 10   & 20   & 20   & 20   & 20   \\
Time (s) &3.2 & 1.21 & 1.14 & 1.05 & 0.99 & 1.07 & 1.56 \\ \hline
\end{tabular}
\end{table}

\section{Conclusions}
This paper introduced a new methodology for solutions of problems of
scattering by periodic arrays of cylinders with applicability
throughout the spectrum---even at and around Wood frequencies. To the
best of our knowledge, this is the first particle-array
periodic-Green-function method that remains applicable around Wood
frequencies. The algorithm yields fast solutions for frequencies in
the resonance region, where wavelengths are comparable to the
structural period. High-frequency problems are also amenable to
efficient treatment by these methods in conjunction FFT-based
acceleration approaches~\cite{BrunoKunyansky,BrunoShipman}; the
development of such accelerated methods, however, is left for future
work.
\section*{Acknowledgements}
The authors gratefully acknowledge support by NSF and AFOSR through
contracts DMS-1411876 and FA9550-15-1-0043, and by the NSSEFF Vannevar
Bush Fellowship under contract number N00014-16-1-2808.

\appendix
\section{Rayleigh expansion of $v_{k_0}$ for $k_0\in \mathcal{K}_{\mathrm{wa}}$}
\label{app:Rayleigh}
The Rayleigh expansion of $v_{k_0}$ can be obtained by substituting
the Rayleigh series of the shifted Green function $G_{j,k_0}^q$ for
$k_0\in \mathcal{K}_{\mathrm{wa}}$ in equation~\eqref{eq:potVk_0}. In
the case $j\geq 2$ we have~\cite[eqs. (54), (56)]{BrunoDelourme}
\begin{equation}
G_{j,k_0}^q(X,Y)=\frac{i}{2L} \sum \limits_{n\not \in \mathcal{N}_{\mathrm{wa}}} \frac{(1-e^{i\beta_n(k_0)})^j}{\beta_n(k_0)} e^{i\alpha_n(k) X+i\beta_n(k) Y}
\end{equation}
for $Y>0$, while, for $-h<Y<0$, the corresponding modified version of
equation~\cite[eq. (54)]{BrunoDelourme} yields
\begin{equation}
G_{j,k_0}^q(X,Y)=\frac{i}{2L} \sum \limits_{n\not \in \mathcal{N}_{\mathrm{wa}}} \frac{e^{i\alpha_n(k) X-i\beta_n(k) Y}}{\beta_n(k_0)} +  \frac{(1-e^{i\beta_n(k_0)})^j-1}{\beta_n(k_0)} e^{i\alpha_n(k) X+i\beta_n(k) Y} 
+\frac{1}{L}\sum \limits_{n\in \mathcal{N}_{\mathrm{wa}}}  Y e^{i\alpha_n(k_0)X }.
\end{equation}
The analogous expressions for $j=1$ are
\begin{equation}
G_{j,k_0}^q(X,Y)=\frac{i}{2L} \sum \limits_{n\not \in \mathcal{N}_{\mathrm{wa}}} \frac{(1-e^{i\beta_n(k_0)})^j}{\beta_n(k_0)} e^{i\alpha_n(k) X+i\beta_n(k) Y}+\frac{h}{2L}\sum \limits_{n \in \mathcal{N}_{\mathrm{wa}}}  e^{i\alpha_n(k_0) X}
\end{equation}
for $Y>0$ and
\begin{equation}
G_{j,k_0}^q(X,Y)=\frac{i}{2L} \sum \limits_{n\not \in \mathcal{N}_{\mathrm{wa}}} \frac{e^{i\alpha_n(k) X-i\beta_n(k) Y}}{\beta_n(k_0)} +  \frac{(1-e^{i\beta_n(k_0)})^j-1}{\beta_n(k_0)} e^{i\alpha_n(k) X+i\beta_n(k) Y} 
+\frac{1}{L}\sum \limits_{n\in \mathcal{N}_{\mathrm{wa}}} (\frac{h}{2}+ Y) e^{i\alpha_n(k_0)X } 
\end{equation}
for $-h<Y<0$.  For $j\geq 2$ we thus obtain
\begin{equation}
\label{RayleighAbove}
v_{k_0}(x,y)= \frac{i}{2L}\sum\limits_{n\not \in \mathcal{N}_\mathrm{wa}} \frac{I_{n,k_0}^+[\psi]}{\beta_n(k)} e^{i\alpha_n(k) x+i \beta_n(k) y}
\end{equation}
in the region $\{y>M^+\}$ and
\begin{equation}
\label{RayleighBelow}
v_{k_0}(x,y)= \frac{i}{2L}\sum\limits_{n\not \in \mathcal{N}_\mathrm{wa}} \frac{I_{n,k_0}^-[\psi]}{\beta_n(k_0)} e^{i\alpha_n(k_0) x-i \beta_n(k_0) y} + \frac{1}{L}\sum\limits_{n  \in \mathcal{N}_\mathrm{wa}} I_{n,k_0}^+[\psi] y e^{i\alpha_n(k_0) x}  - \tilde{I}_{n,k_0}[\psi] e^{i\alpha_n(k_0) x}
\end{equation}
in the region $\{y<M^-\}$ (see Figure~\ref{problemGeometry}). Here,
the functional $\tilde{I}_{n,k_0}$ is defined in
equation~\eqref{eq:functionalITilde}; note, further, that for $n \in
\mathcal{N}_\mathrm{wa}$ we have $I^+_{n,k_0}=I^-_{n,k_0}$.

In the case $j=1$, finally, we obtain
\begin{equation}
\label{RayleighAbove_1}
v_{k_0}(x,y)= \frac{i}{2L}\sum\limits_{n\not \in \mathcal{N}_\mathrm{wa}} \frac{I_{n,k_0}^+[\psi]}{\beta_n(k)} e^{i\alpha_n(k) x+i \beta_n(k) y} +\frac{h}{2L}\sum\limits_{n \in \mathcal{N}_\mathrm{wa}} I^+_{n,k_0}[\psi] e^{i\alpha_n(k_0) x }
\end{equation}
in the region $\{y>M^+\}$ and
\begin{equation}
\label{RayleighBelow_1}
v_{k_0}(x,y)= \frac{i}{2L}\sum\limits_{n\not \in \mathcal{N}_\mathrm{wa}} \frac{I_{n,k_0}^-[\psi]}{\beta_n(k_0)} e^{i\alpha_n(k_0) x-i \beta_n(k_0) y} + \frac{1}{L}\sum\limits_{n  \in \mathcal{N}_\mathrm{wa}} I_{n,k_0}^+[\psi] y e^{i\alpha_n(k_0) x}  +\left(\frac{h}{2} I_{n,k_0}^+[\psi]- \tilde{I}_{n,k_0}[\psi]\right) e^{i\alpha_n(k_0) x}
\end{equation}
in the region $\{y<M^-\}$ (see Figure~\ref{problemGeometry}).

\begin{remark}\label{rem:invert_Ak}
  From equations~\eqref{RayleighBelow} and~\eqref{RayleighBelow_1} we
  see that that the potential $v_{k_0}$ could in principle contain
  unbounded modes (that grow linearly with $y$) in the region
  $\{y<M^-\}$---and, thus, the solution $v_{k_0}$ could in principle
  fail to satisfy the radiation condition put forth in
  Remark~\ref{rem:Radiating}. The unbounded modes are not present in
  the expansions~\eqref{RayleighBelow} and~\eqref{RayleighBelow_1}
  provided the density $\psi$ satisfies the conditions
  $I_{n,k_0}^+[\psi]=0$ for all $n \in \mathcal{N}_\mathrm{wa}$.  But,
  in view of equation~\eqref{eq:d_coeffs} and the parenthetical
  comment that follows~\eqref{eq:finite_lc}, equation~\eqref{lim_plus}
  tells us that $I_{n_j,k}^+$ ($j=1,\dots,r$) vanishes for
  $k=k_0$---since so does the denominator on the left hand side of
  that equation. Since $n_j$ ($j=1,\dots,r$) is an enumeration of $
  \mathcal{N}_\mathrm{wa}$ we see that the density $\psi_{k_0}$
  obtained at a Wood anomaly in Theorem~\ref{thm:Woodbury} indeed
  satisfies the condition $I_{n,k_0}^+[\psi]=0$ for all $n \in
  \mathcal{N}_\mathrm{wa}$. It follows that the corresponding
  potential $v_{k_0}$ (defined in terms of $\psi_{k_0}$) is a
  radiating potential both above and below the obstacles.
\end{remark}
\begin{remark}\label{rem:non-radiating}
  In connection with the previous remark we note that the functionals
  $I_{n,k_0}^+[\psi]$ ($n \in \mathcal{N}_\mathrm{wa}$) do not
  identically vanish: e.g., setting $\psi = e^{i\alpha_n x}$
  in~\eqref{InPlus} yields a non-zero value, since the integral
  involving the normal derivative does vanish. It follows that, for a
  given (arbitrary) density $\psi$, the potential $v_{k_0}$ does
  not generally satisfy the radiation condition and, therefore, as
  indicated in Section~\ref{sec:Woodbury}, classical arguments based
  on uniqueness of radiating solution for the associated PDE problem
  are not immediately applicable to the study of uniqueness of the
  operator $A_{k_0}$ with $k_0\in\mathcal{K}_\mathrm{wa}$.
\end{remark}
\section{Invertibility of $\tilde{R}_{k_0} A_{k_0}^{-1}$ }
\label{app:invertibility}
Using the representations~\eqref{RayleighAbove}
and~\eqref{RayleighBelow} (in the case $j\geq 2$)
or~\eqref{RayleighAbove_1} and~\eqref{RayleighBelow_1} (in the case
$j=1$) we can now deduce the invertibility of the restriction of the
operator $\tilde{R}_{k_0}A_{k_0}^{-1}$ to the finite-dimensional space
$S$ (see Theorem~\ref{thm:Woodbury}) provided uniqueness of solution
of the problem~\eqref{PDE}--\eqref{TE} holds for $k=k_0$. Indeed, for $f \in S$ such
that $\tilde{R}_{k_0}A_{k_0}^{-1} f =0$, in view of
equation~\eqref{RDirichlet2} we must have
$I_{n,k_0}^+[A_{k_0}^{-1}f]=0$ for all $n \in
\mathcal{N}_{\mathrm{wa}}$. In view of Remark~\ref{rem:invert_Ak},
then, the potential $v_{k_0}$ defined in~\eqref{eq:potVk_0} with
$\psi=A_{k_0}^{-1}f$ is a radiating solution of the Helmholtz
equation. Moreover, using the jump-relations of the single and double
layer potentials it follows that $v_{k_0}$ equals $f$ along the
boundary of $D_0$. But, since $f\in S$,
\begin{equation}
f =\sum \limits_{n\in \mathcal{N}_{\mathrm{wa}}} C_n e^{i\alpha_n(k) x},
\end{equation}
it follows (assuming the uniqueness of radiating solutions for the PDE
problem for $k=k_0$) that
\begin{equation}
  v_k(x,y) = \sum \limits_{n\in \mathcal{N}_{\mathrm{wa}}} C_n e^{i\alpha_n(k) x}
\end{equation} 
in the complete exterior domain $\Omega=\RR^2 \setminus D$. Since
$I_{n,k_0}^+[\psi]=0$, $n\in \mathcal{N}_{\mathrm{wa}}$ then $v_k$
does not contain the Wood modes in the region $y>M^+$ (see
equations~\eqref{RayleighAbove} or~\eqref{RayleighAbove_1} in the
respective cases $j\geq 2$ and $j=1$) and therefore $C_n=0$ for all
$n\in \mathcal{N}_\mathrm{wa}$. In other words, the operator
$\tilde{R}_{k_0}A_{k_0}^{-1}:S \to S$ in finite-dimensional vector
space $S$ has trivial null space, and it must therefore be invertible,
as desired.  

\bibliographystyle{plain} 
\bibliography{references}

\end{document}